%% file: main.tex
\documentclass[a4paper]{article}
\usepackage[T1]{fontenc}
\usepackage{lmodern}
\usepackage{amsmath,amssymb}
\usepackage{amsthm}
\usepackage{mathrsfs}
\usepackage[all]{xy}
\usepackage{rotate}
\usepackage{enumitem}
\usepackage{mathtools}
\usepackage[dvipdfmx]{graphics}

\usepackage{here}

\input{newcommands.tex}

\begin{document}

\title{Cartan subalgebras for restrictions of $\mathfrak{g}$-modules}

\author{Masatoshi Kitagawa}

\date{}

\maketitle

\input{introduction.tex}

\input{preliminary.tex}

\input{general.tex}

\input{branching.tex}

\bibliographystyle{abbrv}


\input{ref.bbl}

\end{document}

%% file: newcommands.tex
\theoremstyle{plain}
\newtheorem{theorem}{Theorem}[section]
\newtheorem*{theorem*}{Theorem}
\newtheorem{proposition}[theorem]{Proposition}
\newtheorem{lemma}[theorem]{Lemma}
\newtheorem{corollary}[theorem]{Corollary}
\newtheorem{conjecture}[theorem]{Conjecture}
\newtheorem{fact}[theorem]{Fact}

\theoremstyle{definition}
\newtheorem{definition}[theorem]{Definition}
\newtheorem{example}[theorem]{Example}

\newtheorem{remark}[theorem]{Remark}

\numberwithin{equation}{subsection}

\newcommand{\define}[1]{\textit{#1}}

\newcommand{\spn}[1]{\mathrm{span}_{\CC}\{#1\}}
\newcommand{\set}[1]{\left\{#1 \right\}}

\newcommand{\Ker}{\mathrm{Ker}}

\newcommand{\Hom}{\mathrm{Hom}}
\DeclareMathOperator{\End}{End}
\DeclareMathOperator{\Aut}{Aut}
\DeclareMathOperator{\id}{id}

\newcommand{\BasicRing}[1]{\mathbb{#1}}
\newcommand{\ZZ}{\BasicRing{Z}}
\newcommand{\NN}{\BasicRing{N}}

\newcommand{\RR}{\BasicRing{R}}
\newcommand{\CC}{\BasicRing{C}}

\newcommand{\Ann}{\mathrm{Ann}}

\newcommand{\alg}[1]{\mathcal{#1}}

\DeclareMathOperator{\gr}{gr}
\DeclareMathOperator{\Ass}{Ass}

\DeclareMathOperator{\rank}{rank}
\newcommand{\Ad}{\mathrm{Ad}}
\newcommand{\ad}{\mathrm{ad}}
\newcommand{\classicalG}[1]{\mathrm{#1}}
\newcommand{\classicalg}[1]{\mathfrak{#1}}
\DeclareMathOperator{\Lie}{Lie}

\newcommand{\lieC}[1]{\mathfrak{#1}}
\newcommand{\lie}[1]{\mathfrak{#1}}

\newcommand{\univ}[1]{\mathcal{U}(\lieC{#1})}
\newcommand{\univfilt}[2]{\mathcal{U}_{#1}(\lieC{#2})}
\newcommand{\univcent}[1]{\mathcal{Z}(\lieC{#1})}

\newcommand{\AV}{\mathcal{AV}}
\newcommand{\hsum}{\sideset{}{^\oplus}{\sum}}

\newcommand{\hrep}[1]{\mathcal{#1}}

\newcommand{\orbit}[1]{\mathcal{#1}}

\newcommand{\Dsheaf}[2]{\mathscr{#1}_{#2}}

\newcommand{\Dalg}[2]{\mathcal{#1}_{#2}}
\newcommand{\ntDalg}[1]{\Dalg{D}{#1}}

\newcommand{\sect}{\Gamma}

\newcommand{\regular}{\mathcal{O}}
\newcommand{\rring}[1]{\regular(#1)}

\newcommand{\Variety}{\mathcal{V}}

\DeclareMathOperator{\Rrankmin}{loc\text{.}\underline{rank}}
\DeclareMathOperator{\Rrankmax}{loc\text{.}\overline{rank}}
\DeclareMathOperator{\Rrank}{loc\text{.}rank}
\DeclareMathOperator{\glrank}{gl\text{.}rank}
\DeclareMathOperator{\Krankmin}{K\text{.}\underline{rank}}
\DeclareMathOperator{\Krankmax}{K\text{.}\overline{rank}}
\DeclareMathOperator{\Krank}{K\text{.}rank}

\let\Re\relax
\let\Im\relax
\DeclareMathOperator{\Re}{\mathrm{Re}}
\DeclareMathOperator{\Im}{\mathrm{Im}}
\DeclareMathOperator{\res}{\mathrm{rest}}
\newcommand{\GIT}{/\!/}
\newcommand{\rfilt}[2]{\mathcal{R}_{#1}(#2)}

%% file: introduction.tex
\begin{abstract}
	In this paper, we deal with the $\mathcal{U}(\mathfrak{g})$-action on a $\mathfrak{g}$-module on which a larger algebra $\mathcal{A}$ acts irreducibly.
	Under a mild condition, we will show that the support of the $\mathcal{Z}(\mathfrak{g})$-action is a union of affine subspaces in the dual of a Cartan subalgebra modulo the Weyl group action.
	As a consequence, we propose a definition of a Cartan subalgebra for such a $\mathfrak{g}$-module.

	The support of the $\mathcal{Z}(\mathfrak{g})$-module is an algebraic counterpart of the support of the measure in the irreducible decomposition of a unitary representation.
	This consideration is motivated by the theory of the discrete decomposability initiated by T.\ Kobayashi.
	Defining a Cartan subalgebra for a $\mathfrak{g}$-module is motivated by the study of I.\ Losev on Poisson $G$-varieties.
	These are related each other through the associated variety and the nilpotent orbit associated to a $\mathfrak{g}$-module.
\end{abstract}

\section{Introduction}

The purpose of this paper is to study the $\univcent{g}$-action on a $\lie{g}$-module on which a larger algebra acts irreducibly.
The restriction to $\lie{g}$ of modules of larger Lie algebras and algebras of twisted differential operators are the main objects.
A definition of Cartan subalgebras of such restrictions is proposed in this paper.

Let $G_\RR$ be a connected reductive Lie group and $\hrep{H}$ a unitary representation of $G_\RR$ on a separable Hilbert space.
By the theorem of Mautner--Teleman, $\hrep{H}$ has a unique irreducible decomposition by direct integral as
\begin{align}
	\hrep{H} \simeq \int_{\widehat{G}_\RR}^\oplus m(\pi) \hrep{H}_\pi d\mu(\pi), \label{intro:eqn:IrreducibleDecomp}
\end{align}
where $\mu$ is a Borel measure on the unitary dual $\widehat{G}_\RR$, and $m\colon \widehat{G}_\RR\rightarrow \NN \cup \set{\infty}$  is the multiplicity function.
If the representation $\hrep{H}$ has no additional condition, the decomposition has no restriction, so it may be too wild in general.
In practice, we often impose strong assumptions on $\hrep{H}$, e.g.\ $\hrep{H} = L^2(X)$ for some $G_\RR$-manifold $X$ or $\hrep{H} = \hrep{W}|_{G_\RR}$ the restriction of an irreducible representation of a larger Lie group.
Hence the irreducible decomposition of $\hrep{H}$ is expected to be controllable.

We focus on `the largest components' of the support of the measure $\mu$.
In this paper, we mainly study an algebraic counterpart.
Let $\lie{g}$ be a finite-dimensional complex reductive Lie algebra.
Fix a Cartan subalgebra $\lie{h}$ of $\lie{g}$ and write $W$ for the Weyl group of $\lie{g}$.
For a non-zero $\lie{g}$-module $V$, we define (Definition \ref{def:Rank})
\begin{align*}
	\rfilt{k}{V} &\coloneq \set{v \in V: \dim(\Variety(\Ann_{\univcent{g}}(v))) \leq k} \quad (k\geq -1), \\
	\Rrankmin(V) &\coloneq \min\set{k \in \NN : \rfilt{k}{V} \neq 0}, \\ 
	\Rrankmax(V) &\coloneq \min\set{k \in \NN : \rfilt{k}{V} = V},
\end{align*}
where $\univcent{g}$ is the center of the universal enveloping algebra $\univ{g}$ and $\Variety(I)$ denotes the subvariety of $\lie{h}^*/W$ corresponding to an ideal $I\subset \univcent{g} \simeq S(\lie{h})^W$.
If $\Rrankmin(V) = \Rrankmax(V)$, we write $\Rrank(V)\coloneq \Rrankmin(V)$, and say that $V$ is equi-rank.
We consider the support of the $\univcent{g}$-module $V$ as an algebraic counterpart of the support of the measure $\mu$.

We shall state the main results about the shape of the variety $\Variety(\Ann_{\univcent{g}}(v))$.
Let $(\alg{A}, G)$ be a generalized pair with connected reductive algebraic group $G$ (see Definition \ref{def:GeneralizedPair}).
$(\univ{\widetilde{g}}, G)$ and $(\ntDalg{X}, G)$ are typical examples, where $\lie{\widetilde{g}}$ is a Lie algebra containing $\lie{g}$ and $X$ is a $G$-variety.
It is important for us that the adjoint action of $\lie{g}$ on $\alg{A}$ is locally finite and completely reducible.
Let $q\colon \lie{h}^* \rightarrow \lie{h}^*/W$ denote the quotient map.
The following two theorems are the main results in this paper.
See Theorems \ref{thm:RankFiltration} and \ref{thm:Affinity}.

\begin{theorem}\label{intro:thm:EquiRank}
	For any irreducible $\alg{A}$-module, the restriction $V|_{\lie{g}}$ is equi-rank.
\end{theorem}

\begin{theorem}\label{intro:thm:Affinity}
	Let $V$ be an irreducible $\alg{A}$-module.
	Suppose that $\alg{A}$ has at most countable dimension.
	Write $R$ for the set of all $\lie{h}$-weights in $\alg{A}$.
	Then there exist a subspace $\lie{a}^*\subset \lie{h}^*$ and $\lambda_0 \in \lie{h}^*$
	satisfying the following properties.
	\begin{enumerate}
		\item $\lie{a}^*$ is spanned by some elements of the form $\mu - \mu'$ ($\mu, \mu' \in R$).
		\item For any $P \in \Ass_{\univcent{g}}(V|_{\lie{g}})$, there exists $\alpha \in R$ such that $\Variety(P) = q(\lie{a}^* + \lambda_0 + \alpha)$.
	\end{enumerate}
\end{theorem}

To show Theorem \ref{intro:thm:Affinity}, the assumption that $V$ has at most countable dimension is essential.
We will give an example in Remark \ref{rmk:Affinity}.
Using Theorem \ref{intro:thm:Affinity}, we define a small Cartan subalgebra for $V$ by the dual of $\lie{a}^*$ (i.e.\ $\lie{h}/\bigcap_{\lambda \in \lie{a^*}} \Ker(\lambda)$).
See Definition \ref{def:Cartan}.
If $G$ is a complex torus, a stronger result than Theorem \ref{intro:thm:Affinity} holds,
which is a generalization of weight space decomposition.
We will show this in Theorem \ref{thm:DirectSumAbelian}.

Hereafter, assume that $\alg{A}$ has at most countable dimension.
One of the motivations of our study is the theory of discrete decomposability.
The theory is initiated by T.\ Kobayashi in the series of papers \cite{Ko94, Ko98_discrete_decomposable_2, Ko98_discrete_decomposable_3}.
We say that a $\lie{g}$-module $V$ is discretely decomposable if $V$ has an exhaustive filtration $F$ such that $F_{-1}(V) = 0$ and $F_k(V)$ has finite length for any $k$.
Let $K$ be a finite covering of the connected component of a symmetric subgroup of $G$.
It is easy to see that an irreducible $(\alg{A}, K)$-module $V$ is discretely decomposable as a $(\lie{g}, K)$-module if and only if $\Rrank(V|_{\lie{g}}) = 0$.
Theorem \ref{intro:thm:EquiRank} is a generalization of \cite[Lemma 1.5]{Ko98_discrete_decomposable_3}.

In \cite[Theorem 3.1]{Ko98_discrete_decomposable_3}, T.\ Kobayashi gave a necessary condition for the discretely decomposable restriction by the associated variety of a module.
As a generalization of the result, we will give a lower bound of $\Rrank(V)$ by the associated variety.
Suppose that $\alg{A}$ is a filtered $\CC$-algebra such that $\gr \alg{A}$ is a finitely generated integral domain and the homomorphism $\univ{g}\rightarrow \alg{A}$ is a homomorphism of filtered algebras.
Write $X$ for the affine variety corresponding to $\gr \alg{A}$
and $\res_{\lie{g}}\colon X\rightarrow \lie{g}^*$ for the morphism corresponding to $S(\lie{g})\rightarrow \gr \alg{A}$.

We denote by $\widetilde{q}\colon \lie{g}^* \rightarrow \lie{g}^*\GIT G (\simeq \lie{h}^* \GIT W)$
and $q\colon \lie{h}^* \rightarrow \lie{h}^* // W$ the quotient morphisms.
For a finitely generated $\alg{A}$-module $V$, let $\AV(V)$ denote the associated variety of $V$.

\begin{theorem}\label{intro:thm:AV}
	Let $V$ be an irreducible $(\alg{A}, K)$-module.
	\begin{enumerate}
		\item $\AV(\univ{g}v)$ does not depend on the choice of $0\neq v \in V$. (Set $\AV(V|_{\lie{g}})\coloneq \AV(\univ{g}v)$.)
		\item $\overline{\widetilde{q}(\AV(V|_{\lie{g}}))} = q(\lie{a}^*)$, where $\lie{a}^*$ is the dual of a small Cartan subalgebra for $V$.
		\item $\res_{\lie{g}}(\AV(V)) \subset \AV(V|_{\lie{g}})$.
		\item $\dim(\overline{\res_{\lie{g}}(\AV(V))}\GIT K) \leq \Rrank(V|_{\lie{g}})$.
	\end{enumerate}
\end{theorem}

The proofs of 1 and 3 in Theorem \ref{intro:thm:AV} are essentially the same as \cite[Theorems 3.7 and 3.1]{Ko98_discrete_decomposable_3}.
It is a hard problem when $\overline{\res_{\lie{g}}(\AV(V))} = \AV(V|_{\lie{g}})$ holds.
In the context of discrete decomposable restriction of modules of reductive Lie groups,
this equality is conjectured in \cite[Conjecture 5.11]{Ko11}.
For real rank one symmetric subgroups, the equality $\dim(\overline{\res_{\lie{g}}(\AV(V))}\GIT K) = \Rrank(V|_{\lie{g}})$ will be proved in Proposition \ref{prop:Rank1ToConj}.

We shall consider the branching problem of reductive Lie groups.
Let $\widetilde{G}_\RR$ be a connected semisimple Lie group with finite center,
and $G_\RR$ a connected reductive subgroup of $\widetilde{G}_\RR$.
Fix a Cartan involution $\theta$ of $\widetilde{G}_\RR$ that stabilizes $G_\RR$
and write $\widetilde{K}$ (resp.\ $K$) for the complexification of $\widetilde{G}^\theta_\RR$ (resp.\ $G^\theta_\RR$).
Suppose that the Zariski closure $G\coloneq \overline{\Ad(G_\RR)}$ in $\Aut(\lie{\widetilde{g}})$ has the Lie algebra $\lie{g} = \Lie(G_\RR)\otimes_{\RR} \CC$.

We need to modify the definition of the ranks for continuous representations.
For a continuous representation $V$ of $G_\RR$ on a quasi-complete locally convex space,
we define (Definition \ref{def:RankContinous})
\begin{align*}
	\Rrankmin(V) &\coloneq \min\set{k \in \NN : \rfilt{k}{V^\infty} \neq 0}, \\
	\Rrankmax(V) &\coloneq \min\set{k \in \NN : \overline{\rfilt{k}{V^\infty}} = V},
\end{align*}
where $V^\infty$ is the space of all smooth vectors in $V$.
If the support of the measure $\mu$ in \eqref{intro:eqn:IrreducibleDecomp} is a union of algebraic sets in a suitable sense, the ranks means the size of the smallest and the largest components of the support.
It is easy to see
\begin{align*}
	\Rrankmax(\hrep{H}|_{G_\RR}) \leq \Rrankmax(\hrep{H}^\infty|_{G_\RR}) \leq \Rrank(\hrep{H}_{\widetilde{K}}|_{\lie{g}, K})
\end{align*}
for any irreducible unitary representation of $\hrep{H}$ of $\widetilde{G}_\RR$ (see Proposition \ref{prop:BranchingEasy}).
$\Rrank(\hrep{H}_{\widetilde{K}}|_{\lie{g}, K}) = 0$ implies that
$\Rrankmax(\hrep{H}|_{G_\RR}) = 0$ and hence $\hrep{H}|_{G_\RR}$ is discretely decomposable.
This is a paraphrase of \cite[Theorem 2.7]{Ko00_discretely_decomposable} by the ranks.
We will prove fundamental results for the ranks in Section \ref{section:Branching}.

We will show a similar result as Theorem \ref{intro:thm:Affinity} for $\Variety(\Ann_{\univcent{g}}(V))$.
We need to assume that $G_\RR$ is real spherical in $\widetilde{G}_\RR$, that is, $G_\RR \widetilde{P}_\RR$ is open in $\widetilde{G}_\RR$ for some minimal parabolic subgroup $\widetilde{P}_\RR\subset \widetilde{G}_\RR$.

\begin{theorem}\label{intro:thm:AffinityAnn}
	Let $V$ be an irreducible $(\widetilde{\lie{g}}, \widetilde{K})$-module.
	Then there exist a subspace $\lie{c}^* \subset \lie{h}^*$ and $\alpha_1, \ldots, \alpha_r \in \lie{h}^*$ such that $\Variety(\Ann_{\univcent{g}}(V)) = \bigcup_i q(\lie{c}^* + \alpha_i)$.
\end{theorem}

We will prove Theorem \ref{intro:thm:AffinityAnn} in Theorem \ref{thm:AffinityRealSpherical}.
For technical reasons, we need to find an irreducible $\widetilde{\lie{g}}$-module $V'$ such that $\Ann_{\univ{\widetilde{g}}}(V) = \Ann_{\univ{\widetilde{g}}}(V')$ and $V'|_{\lie{g}}$ is finitely generated.
To find $V'$, we use Yamashita's criterion \cite{Ya94} and the real sphericity.

At the last of this section, we shall note a conjecture about the Cartan subalgebra.
In \cite{Lo09}, a Cartan subalgebra for a Poisson $G$-variety is given.
Theorem \ref{intro:thm:AffinityAnn} is regarded as a representation-theoretic analogue of this result.
This is also one of motivations of our study.
It is well-known that $\Variety(\gr \Ann_{\univ{\widetilde{g}}}(V))$ is the closure of one nilpotent coadjoint orbit $\orbit{O}$ in $\widetilde{\lie{g}}^*$ (see \cite[Theorem 9.3]{Ja04} and references therein).
Then $\res_{\lie{g}}\colon \overline{\orbit{O}} \rightarrow \lie{g}^*$ is the moment map of the Poisson $G$-variety $\overline{\orbit{O}}$.
In Proposition \ref{prop:AssociatedVarietyAnn}, we will prove $\overline{\res_{\lie{g}}(\overline{O})}//G \subset q(\lie{c}^*)$.
I.\ Losev proved in \cite[Theorem 1.2.3]{Lo09} that $\overline{\res_{\lie{g}}(\overline{O})}//G = q(\lie{d}^*)$ for some subspace $\lie{d}^*\subset \lie{h}^*$.
From the results, we expect $q(\lie{c}^*) = q(\lie{d}^*)$.

There exist several studies of Cartan subalgebras and Weyl groups for $G$-varieties and Poisson $G$-varieties.
For the Cartan subalgebras for $G$-varieties, see \cite[Chapter 4]{Ti11} and references therein.
Motivated by the studies, we guess that there exists a suitable Weyl group for the small Cartan subalgebra for $V$ and they are constructed geometrically from a Lagrangian subvariety in $\orbit{O}$.

This paper is organized as follows.
In Section 2, we recall the notions of generalized pairs and infinitesimal characters.
We also treat existence of irreducible quotients of $G_\RR$-representations.
The Kostant's theorem (Fact \ref{fact:Kostant}) is the key fact in this paper.
In Section 3, we introduce the ranks and prove Theorems \ref{intro:thm:EquiRank} and \ref{intro:thm:Affinity}.
Section 4 is devoted to the branching problem of reductive Lie groups.
Theorems \ref{intro:thm:AV} and \ref{intro:thm:AffinityAnn} are proved here.

\subsection*{Notation and convention}

In this paper, any algebra $\alg{A}$ except Lie algebras is unital, associative and over $\CC$.
We express real Lie groups and their Lie algebras by Roman alphabets and corresponding German letters with subscript $(\cdot)_\RR$, and express complex Lie groups (or affine algebraic groups) and their Lie algebras by Roman alphabets and corresponding German letters, respectively.
Similarly, we express the complexification of a real Lie algebra by the same German letter as that of the real form without any subscript.
For example, the Lie algebras of real Lie groups $G_\RR, K_\RR$ and $H_\RR$ are denoted as $\lie{g}_\RR, \lie{k}_\RR$ and $\lie{h}_\RR$ with complexifications $\lie{g}$, $\lie{k}$ and $\lie{h}$, respectively.
For a topological group $G$, let $G_0$ denote the identity component of $G$.

For a $G$-set $X$ of a group $G$, we write $X^G$ for the subset of all $G$-invariant elements in $X$.
We use similar notations for the set of all vectors annihilated by a Lie algebra $\lie{g}$ (resp.\ an ideal $I$) as $V^{\lie{g}}$ (resp.\ $V^I$).
The coordinate ring of an affine variety $X$ is denoted by $\rring{X}$.
For an ideal $I$ of $\rring{X}$, we denote by $\Variety(I)$ the subvariety in $X$ determined by $I$.

\subsection*{acknowledgement}

This work was supported by JSPS KAKENHI Grant Number JP23K12963.

%% file: preliminary.tex
\section{Preliminary}

In this section, we recall basic results about generalized pair, infinitesimal character and irreducible quotient.

\subsection{Generalized pair}

In this subsection, we recall the notion of generalized pairs.
We refer the reader to \cite[p.96]{KnVo95_cohomological_induction}.

In this paper, any variety is defined over $\CC$.
Let $G$ be a connected reductive algebraic group.
For a representation $V$ of $G$ as an abstract group,
we say that $V$ is a $G$-module or $G$ acts rationally on $V$ if the $G$-action on $V$ is locally finite
and any finite-dimensional subrepresentation of $V$ is a representation as an algebraic group.
Hence any $G$-module is completely reducible.

\begin{definition}\label{def:GeneralizedPair}
	Let $\alg{A}$ be a unital associative $\CC$-algebra.
	Suppose that $G$ acts on $\alg{A}$ rationally by algebra automorphisms.
	We say that $(\alg{A}, G)$ equipped with a $G$-equivariant algebra homomorphism $\iota\colon \univ{g} \rightarrow \alg{A}$ is a \define{generalized pair} if the adjoint action of $\lie{g}$ on $\alg{A}$
	given by $\ad_{\alg{A}}(X)a = [\iota(X), a]$ ($X \in \lie{g}, a \in \alg{A}$)
	coincides with the differential of the action of $G$ on $\alg{A}$.
\end{definition}

For a generalized pair $(\alg{A}, G)$, we denote by $\Ad_{\alg{A}}$ (or $\Ad$) the action of $G$ on $\alg{A}$.
If $V$ is an $\alg{A}$-algebra, then $V$ is a $\lie{g}$-module via the homomorphism $\iota\colon \univ{g}\rightarrow \alg{A}$.
We denote by $V|_{\lie{g}}$ the $\lie{g}$-module.
When we write the action of $\lie{g}$ on $V|_{\lie{g}}$, we omit $\iota$ as $Xv$ ($X \in \lie{g}, v \in V$).

We give typical examples of generalized pairs.
If $G'$ is a connected reductive subgroup of $G$, then $(\univ{g}, G')$ is a generalized pair.
If $\Dsheaf{A}{X}$ is a $G$-equivariant algebra of twisted differential operators on a smooth $G$-variety $X$, then $(\sect(\Dsheaf{A}{X}), G)$ forms a generalized pair.

We only need the following trivial property of generalized pairs.

\begin{proposition}
	Let $(\alg{A}, G)$ be a generalized pair and $V$ an $\alg{A}$-module.
	Then the multiplication map $\alg{A}\otimes V\rightarrow V$ is a $\lie{g}$-homomorphism.
\end{proposition}

\begin{definition}
	Let $(\alg{A}, G)$ be a generalized pair.
	An $\alg{A}$-module $V$ is called $(\alg{A}, G)$-module if the $\lie{g}$-action on $V$ lifts to a rational $G$-action.
\end{definition}

Since we have assumed that $G$ is connected, the rational $G$-action on $V$ is unique.
Note that $(\alg{A}, G)$-modules are not main objects in this paper.
The notion of $(\alg{A}, G)$-modules is used only in Subsection \ref{subsection:AV}.

\subsection{Infinitesimal character}

We recall fundamental facts of infinitesimal characters of $\lie{g}$-modules.
Let $\lie{g}$ be a finite-dimensional reductive Lie algebra over $\CC$.
Fix a Cartan subalgebra $\lie{h}\subset \lie{g}$ and a set $\Delta^+(\lie{g}, \lie{h})$ of positive roots.
Let $W$ denote the Weyl group of $\lie{g}$.
Then the center $\univcent{g}$ of $\univ{g}$ is isomorphic to $S(\lie{h})^W$ via the Harish-Chandra isomorphism.
See \cite[Theorem 3.2.3]{Wa88_real_reductive_I}.

By the Harish-Chandra isomorphism, there is a bijection between the set of characters of $\univcent{g}$ and $\lie{h}^*/W$.
Under this identification, we call an element of $\lie{h}^*/W$ or $\lie{h}^*$ an infinitesimal character of $\lie{g}$.
For $\lambda \in \lie{h}^*$, we write $[\lambda] \in \lie{h}^*/W$ for the image via the quotient map,
and $\chi_{\lambda}$ (or $\chi_{[\lambda]}$) for the corresponding character of $\univcent{g}$.
In this context, $\Ker([\lambda])$ means the kernel of $\chi_{[\lambda]}$, which is a maximal ideal of $\univcent{g}$.

Let $V$ be a $\lie{g}$-module and $\lambda$ a character of $\lie{h}$.
We set
\begin{align*}
	V_{\lambda} ( = V_{[\lambda]} = V_{\chi_\lambda}) \coloneq \set{v \in V : \Ker(\chi_\lambda)^n v = 0 \text{ for some }n}.
\end{align*}
We call $V_\lambda$ the primary component of $V$ with the infinitesimal character $\lambda$.
$V$ is said to be $\univcent{g}$-finite if $\univcent{g}/\Ann_{\univcent{g}}(V)$ is finite dimensional.
By \cite[Proposition 7.20]{KnVo95_cohomological_induction}, if $V$ is $\univcent{g}$-finite, then $V$ is the finite direct sum of all non-zero primary components $V_\chi$, which is called the primary decomposition of $V$.

The following result by B.\ Kostant is essential for our study.
The last assertion can be read from the proof of \cite[Theorem 7.133]{KnVo95_cohomological_induction}.

\begin{fact}[B.\ Kostant (see {\cite[Theorem 7.133]{KnVo95_cohomological_induction}})]\label{fact:Kostant}
	Let $V$ be a $\lie{g}$-module with an infinitesimal character $[\lambda]\in \lie{h}^*/W$
	and $F$ a completely reducible finite-dimensional $\lie{g}$-module.
	Then $V\otimes F$ is $\univcent{g}$-finite and has the primary decomposition
	\begin{align*}
		V\otimes F = \bigoplus_{\chi\in R} (V\otimes F)_{\chi},
	\end{align*}
	where $R\coloneq \set{[\lambda + \nu] : \nu \text{ is a weight of }F}$.
	Moreover, $\Ker(\chi)^{|W|}(V\otimes F)_{\chi} = 0$ for any $\chi \in R$.
\end{fact}

\subsection{Existence of irreducible quotients}

Let $\widetilde{G}_\RR$ be a connected reductive Lie group and $G_\RR$ a connected reductive subgroup of $\widetilde{G}_\RR$.
For an irreducible representation $V$ of $\widetilde{G}_\RR$, it is a non-trivial problem how many irreducible quotients the restriction $V|_{G_\RR}$ has.
We recall two results for this problem.

Let $\hrep{H}$ be a Hilbert representation of $\widetilde{G}_\RR$.
Then $\hrep{H}^\infty$ denotes the space of all smooth vectors in $\hrep{H}$ equipped with the topology given by the seminorms $\set{\|X\cdot\|: X \in \univ{\widetilde{g}}}$.
For a finite-dimensional $\widetilde{G}_\RR$-stable subspace $F\subset \univ{\widetilde{g}}$ containing $1$,
we denote by $\hrep{H}_F$ the completion of $\hrep{H}^\infty$ with respect to the seminorms $\set{\|X\cdot \| : X \in F}$ (i.e.\ the graph norm).
Then $\hrep{H}_F$ is a Hilbert representation of $\widetilde{G}_\RR$.

The following proposition is easy from the definition of the topology of $\hrep{H}^\infty$.

\begin{proposition}
	Let $\hrep{H}$ be a Hilbert representation of $\widetilde{G}_\RR$
	and $T\colon \hrep{H}^\infty\rightarrow \hrep{W}$ a continuous $G_\RR$-linear map to a Hilbert representation $\hrep{W}$ of $G_\RR$.
	Then there exists a finite-dimensional $\widetilde{G}_\RR$-stable subspace $F\subset \univ{\widetilde{g}}$ containing $1$ such that
	$T$ extends to a continuous $G_\RR$-linear map from $\hrep{H}_F$ to $\hrep{W}$.
\end{proposition}

For any irreducible admissible Hilbert representation $\hrep{H}$,
$\hrep{H}^\infty$ is realized as a subrepresentation of a principal series representation by the Casselman subrepresentation theorem \cite[Theorem 3.8.3]{Wa88_real_reductive_I}.
The full flag manifold of $\widetilde{G}_\RR$ contains a closed $G_\RR$-orbit.
Considering the restriction map (and some projection), we can construct an irreducible admissible quotient of $\hrep{H}^\infty|_{G_\RR}$.
See \cite[Lemma 6]{Ki24} for the details.

\begin{proposition}\label{prop:ExistenceQuotient}
	Let $\hrep{H}$ be an irreducible admissible Hilbert representation of $\widetilde{G}_\RR$.
	Then $\hrep{H}^\infty|_{G_\RR}$ has an irreducible admissible quotient.
	Moreover, there exists a finite-dimensional $\widetilde{G}_\RR$-stable subspace $F\subset \univ{\widetilde{g}}$ containing $1$ such that $\hrep{H}_F|_{G_\RR}$ has an irreducible admissible quotient.
\end{proposition}

We shall consider a similar result for unitary representations.
We refer the reader to \cite{Be88_plancherel}, \cite[Chapter 12]{Sc90} and \cite[Subsection 3.3]{Li18}.
Let $\hrep{H}$ be an irreducible unitary representation of $\widetilde{G}_\RR$.
Then the restriction $\hrep{H}|_{G_\RR}$ has the irreducible decomposition
\begin{align*}
	\hrep{H}|_{G_\RR} \simeq \int_{\widehat{G}_\RR}^\oplus m(\pi)\hrep{H}_\pi d\mu(\pi),
\end{align*}
where $\mu$ is a Borel measure on the unitary dual $\widehat{G}_\RR$ and $m\colon \widehat{G}_\RR \rightarrow \NN\cup \set{\infty}$ is the multiplicity function.
Remark that there is no continuous $G_\RR$-linear map from $\hrep{H}|_{G_\RR}$ to $\hrep{H}_\pi$ unless $\hrep{H}_\pi$ is a discrete spectrum of $\hrep{H}|_{G_\RR}$.
On the other hand, similar maps may be defined on $\hrep{H}^\infty|_{G_\RR}$ and some completion.
The following fact is essential for this purpose.
See \cite[Theorem 1.5 and Lemma 1.3]{Be88_plancherel}.

\begin{fact}\label{fact:HilbertSchmidt}
	Let $\varphi\colon \hrep{W} \rightarrow \hrep{H} = \int_X^\oplus 
	\hrep{H}_x d\mu(x)$ be a Hilbert--Schmidt $G_\RR$-equivariant operator from a Hilbert representation of $G_\RR$ to a direct integral of unitary representations of $G_\RR$.
	Suppose that $\hrep{W}$ is separable and $\varphi(\hrep{W})$ is dense in $\hrep{H}$.
	Then there exists a family $\set{\varphi_x : \hrep{W}\rightarrow \hrep{H}_x}_{x \in X}$ of continuous $G_\RR$-linear maps satisfying the following properties.
	\begin{enumerate}
		\item $\varphi_x(\hrep{W})$ is dense in $\hrep{H}_x$ $\mu$-a.e.
		\item For any $v \in \hrep{W}$, $\varphi(v) = (\varphi_x(v))_{x \in X}$ holds as an element of the direct integral $\int_X^\oplus \hrep{H}_x d\mu(x)$.
	\end{enumerate}
\end{fact}

To use Fact \ref{fact:HilbertSchmidt}, we need to find a Hilbert--Schmidt operator.
It is well-known that $\hrep{H}^\infty$ is a nuclear Fr\'echet space for any irreducible unitary representation $\hrep{H}$.
This is a consequence of the Casselman subrepresentation theorem.
In fact, the space of smooth sections of a vector bundle on a flag manifold is nuclear,
and so is its closed subspace.
By the nuclearity, there exists a finite-dimensional $\widetilde{G}_\RR$-subspace $F\subset \univ{\widetilde{g}}$ containing $1$ such that the inclusion $\hrep{H}_F\rightarrow \hrep{H}$ is Hilbert--Schmidt.
Therefore we obtain

\begin{proposition}\label{prop:ExistenceQuotientUnitary}
	Let $\hrep{H}$ be an irreducible unitary representation of $\widetilde{G}_\RR$.
	Then there exist a finite-dimensional $\widetilde{G}_\RR$-stable subspace $F\subset \univ{\widetilde{g}}$ containing $1$ and a family $\set{\hrep{W}_x}_{x\in X}$ of closed subrepresentation of $\hrep{H}_F|_{G_\RR}$ such that $\hrep{H}_F/\hrep{W}_x$ is irreducible admissible for any $x \in X$, and $\bigcap_x \hrep{W}_x = 0$.
\end{proposition}

\begin{remark}
	The family $\set{\hrep{W}_x}_{x\in X}$ can be chosen to satisfy that $\hrep{H}_F/\hrep{W}_x$ is a dense (non-closed) subrepresentation of an irreducible unitary representation for any $x \in X$.
\end{remark}

%% file: general.tex
\section{General theory}

In this section, we consider the restriction of an $\alg{A}$-module to a Lie algebra $\lie{g}$.
We introduce the notions of ranks for $\lie{g}$-modules.
Important properties of the ranks are proved under mild assumptions.
We will treat special results for the branching problem of reductive Lie groups in the next section.

\subsection{Definition of rank}

In this subsection, we define ranks of $\lie{g}$-modules, which intuitively means `the size of continuous spectrum'.
The notion is the main research subject in this paper.
Let $\lie{g}$ be a finite-dimensional reductive Lie algebra over $\CC$.

\begin{definition}\label{def:Rank}
	Let $V$ be a $\lie{g}$-module.
	For each $k \in \NN$, we set
	\begin{align*}
		\rfilt{k}{V}\coloneq \set{v \in V : \dim(\Variety(\Ann_{\univcent{g}}(v))) \leq k}
	\end{align*}
	and call the filtration $0 = \rfilt{-1}{V}\subset\rfilt{0}{V} \subset \cdots $ the \define{rank filtraion} of $V$.
	We define
	\begin{align*}
		\Rrankmin(V)&\coloneq \min\set{k \in \NN : \rfilt{k}{V} \neq 0}, \\
		\Rrankmax(V)&\coloneq \min\set{k \in \NN : \rfilt{k}{V} = V}, \\
		\glrank(V) &\coloneq \dim(\Variety(\Ann_{\univcent{g}}(V)))
	\end{align*}
	if $V\neq 0$, and $\Rrankmin(V) = \Rrankmax(V) = \glrank(V) = -1$ for $V = 0$.
	If $\Rrankmin(V) = \Rrankmax(V)$, we write $\Rrank(V) \coloneq \Rrankmin(V)$
	and we say that $V$ is \define{equi-rank}.
\end{definition}

By definition, the rank filtration is a $\lie{g}$-module filtration and we have
\begin{align*}
	\Rrankmin(V) \leq \Rrankmax(V) \leq \glrank(V) \leq \rank(\lie{g}).
\end{align*}
Although the equality $\Rrankmax(V) = \glrank(V)$ may not hold in general,
it always holds for finitely generated $V$.

\begin{proposition}\label{prop:RankFinitelyGen}
	Let $V$ be a finitely generated $\lie{g}$-module.
	Then $\Rrankmax(V) = \glrank(V)$ holds.
\end{proposition}

\begin{proof}
	We can assume $V\neq 0$.
	Take a finite generating subset $\set{v_1, \ldots, v_n}$ of $V$.
	Then we have $\Ann_{\univcent{g}}(V) = \bigcap_{i} \Ann_{\univcent{g}}(v_i)$.
	This shows the assertion.
\end{proof}

The following properties are clear from the definition of the ranks.

\begin{proposition}\label{prop:RankEasy}
	Let $T\colon W\rightarrow V$ be a $\lie{g}$-homomorphism between non-zero $\lie{g}$-modules.
	\begin{enumerate}
		\item $V$ has a submodule $U$ satisfying $\Rrank(U) = \Rrankmin(V)$.
		\item $T(\rfilt{k}{W}) \subset \rfilt{k}{V}$ for any $k$.
		\item If $T$ is surjective, then
		\begin{align*}
			\Rrankmax(V) \leq \Rrankmax(W),\quad \glrank(V) \leq \glrank(W).
		\end{align*}
		\item If $T$ is injective, then
		\begin{align*}
			\Rrankmin(V) &\leq \Rrankmin(W), \quad \glrank(W) \leq \glrank(V),\\
			\Rrankmax(W) &\leq \Rrankmax(V).
		\end{align*}
		If, moreover, $V$ is equi-rank, then so is $W$ and $\Rrank(W) = \Rrank(V)$.
	\end{enumerate}
\end{proposition}

The rank filtration is characterized as follows.

\begin{proposition}\label{prop:RankFiltr}
	Let $V$ be a $\lie{g}$-module.
	Then $\rfilt{k}{V}/\rfilt{k-1}{V}$ is zero or an equi-rank module with $\Rrank(\rfilt{k}{V}/\rfilt{k-1}{V}) = k$ for any $k$.
	Conversely, an exhaustive filtration $F$ of $V$ (i.e.\ $\bigcup_{i} F_i(V) = V$) satisfies this condition is the rank filtration.
\end{proposition}

\begin{proof}
	Clearly, we may assume $V\neq 0$.
	Let $k \in \NN$ and $0\neq v \in \rfilt{k}{V}$.
	Write $\overline{v}$ for the image of $v$ via the quotient map $\rfilt{k}{V} \rightarrow \rfilt{k}{V}/\rfilt{k-1}{V}$.
	Take a finite generating set $\set{X_1, \ldots, X_r}$ of $\Ann_{\univcent{g}}(\overline{v})$.
	Then we have
	\begin{align*}
		\Ann_{\univcent{g}}(\overline{v}) \cdot \bigcap_{i} \Ann_{\univcent{g}}(X_i v) \subset \Ann_{\univcent{g}}(v).
	\end{align*}
	From $\dim(\Variety(\Ann_{\univcent{g}}(X_i v))) < \dim(\Variety(\Ann_{\univcent{g}}(v)))$, we obtain the desired equality $\dim(\Variety(\Ann_{\univcent{g}}(\overline{v}))) = \dim(\Variety(\Ann_{\univcent{g}}(v)))$.
	This shows the first assertion.

	Take an exhaustive filtration $F$ of $V$ satisfying the property in the first assertion.
	It is clear that $\min\set{k \in \NN : F_k(V)\neq 0} = \Rrankmin(V)$
	and $F_{\Rrankmin(V)}(V) \subset \rfilt{\Rrankmin(V)}{V}$.
	By the assumption of $F$, we have $\Rrankmin(V/F_{\Rrankmin(V)}(V)) > \Rrankmin(V)$.
	This implies $F_{\Rrankmin(V)}(V) \supset \rfilt{\Rrankmin(V)}{V}$.
	Therefore we obtain $F_{\Rrankmin(V)}(V) = \rfilt{\Rrankmin(V)}{V}$.
	Repeating this argument for $V/\rfilt{k}{V}$ inductively, we have shown the second assertion.
\end{proof}

We shall consider a similar notion for continuous representations.
Let $G_\RR$ be a connected reductive Lie group and fix a maximal compact subgroup $K_\RR\subset G_\RR$.
Write $K$ for the complexification of $K_\RR$.
Let $(\pi, V)$ be a continuous representation of $G_\RR$ on a quasi-complete locally convex space $V$.
We denote by $V_K$ (resp.\ $V^\omega$, $V^\infty$) the space of all $K$-finite smooth vectors (resp.\ analytic vectors, smooth vectors) in $V$.
Then $V_K, V^\omega$ and $V^\infty$ are $\lie{g}$-modules.
We abuse notation and write $\pi$ for the $\lie{g}$-actions.

\begin{definition}\label{def:RankContinous}
	We define
	\begin{align*}
		\Rrankmin(V)&\coloneq \min\set{k \in \NN : \rfilt{k}{V^\infty} \neq 0}, \\
		\Rrankmax(V)&\coloneq \min\set{k \in \NN : \overline{\rfilt{k}{V^\infty}} = V}, \\
		\glrank(V) &\coloneq \dim(\Variety(\Ann_{\univcent{g}}(V^\infty))).
	\end{align*}
	If $\Rrankmin(V) = \Rrankmax(V)$, we write $\Rrank(V) \coloneq \Rrankmin(V)$
	and we say that $V$ is \define{equi-rank}.
\end{definition}

The following properties are easy from the definition.
Although they are similar to Proposition \ref{prop:RankEasy}, remark that there is a slight difference.
In the fourth assertion, an analogue of the inequality for $\Rrankmax$ may not hold.

\begin{proposition}\label{prop:RankEasyConti}
	Let $T\colon W\rightarrow V$ be a continuous $G_\RR$-linear map between continuous representations of $G_\RR$ on quasi-complete locally convex spaces $V$ and $W$.
	\begin{enumerate}
		\item $V$ has a closed equi-rank subrepresentation $U$ satisfying $\Rrank(U) = \Rrankmin(V)$.
		\item $T(\rfilt{k}{W^\infty}) \subset \rfilt{k}{V^\infty}$ for any $k$.
		\item If $T$ has dense image, then.
		\begin{align*}
			\Rrankmax(V) \leq \Rrankmax(W), \quad \glrank(V) \leq \glrank(W).
		\end{align*}
		\item If $T$ is injective and $W\neq 0$, then
		\begin{align*}
			\Rrankmin(V) \leq \Rrankmin(W), \quad \glrank(W) \leq \glrank(V).
		\end{align*}
	\end{enumerate}
\end{proposition}

If $\pi(X)$ with domain $V^\infty$ is closable for any $X \in \univcent{g}$, we can define the rank filtration of $V$.
In fact, for any $v \in V$, $\Ann_{\univcent{g}}(v) = \set{X \in \univcent{g} : v \in \Ker(\overline{\pi(X)})}$
is an ideal of $\univcent{g}$.
When we define $\Rrankmin(V)$ and $\Rrankmax(V)$ using this filtration, they coincide with those in Definition \ref{def:RankContinous}.
This follows from $\overline{(V^\infty)^I} = V^I$ for any ideal $I \subset \univcent{g}$.
By the same reason, one can replace $V^\infty$ in Definition \ref{def:RankContinous} with $V_K$ or $V^\omega$.

\begin{example}\label{ex:R}
	Suppose $G_\RR = \RR^2$.
	For $\lambda \in\lie{g}^*$, let $(e^\lambda, \CC_{\lambda})$ denote the corresponding character of $G_\RR$.
	Fix a basis $\set{\lambda_1, \lambda_2}$ of $\lie{g}^*_\RR$.
	We shall consider the ranks of the representation 
	\begin{align*}
		V \coloneq \CC_{0} \oplus \hsum_{n \in \ZZ}\int^\oplus_{\RR} \CC_{\sqrt{-1} (s\lambda_1 + n \lambda_2)} d\mu(s),
	\end{align*}
	where $\mu$ is the Lebesgue measure on $\RR$.
	Then we have
	\begin{align*}
		\Rrankmin(V) = 0, \quad \Rrankmax(V) = 1, \quad \glrank(V) = 2.
	\end{align*}
\end{example}

As in Example \ref{ex:R}, $\Rrankmin(V)$ and $\Rrankmax(V)$ measure the size of some components of spectrum.
In fact, $\Rrankmax(V)$ is related to the size of the largest component of spectrum in unitary representations as follows.

\begin{lemma}\label{lem:EquiRankSub}
	Let $V$ be a unitary representation of $G_\RR$
	and $U$ a non-zero closed subrepresentation.
	Then we have
	\begin{align*}
		\Rrankmin(V) \leq \Rrankmin(U) \leq \Rrankmax(U) \leq \Rrankmax(V).
	\end{align*}
	If, moreover, $V$ is equi-rank, then so is $U$ and $\Rrank(U) = \Rrank(V)$ holds.
\end{lemma}

\begin{proof}
	Applying Proposition \ref{prop:RankEasyConti} to the injection $U\hookrightarrow V$ and the projection $V\rightarrow U$, we obtain the lemma.
\end{proof}

\begin{proposition}
	Let $V$ be a unitary representation of $G_\RR$.
	Suppose that $V = \sum_{i \in I}^\oplus W_i$ is a direct sum of non-zero closed subrepresentations.
	Then we have
	\begin{align*}
		\max\set{\Rrankmax(W_i): i\in I} &= \Rrankmax(V), \\
		\min\set{\Rrankmin(W_i): i\in I} &= \Rrankmin(V).
	\end{align*}
	In particular, $V$ is equi-rank if and only if $W_i$'s are equi-rank and $\Rrank(W_i)$ is constant on $i \in I$.
\end{proposition}

\begin{proof}
	Assume $V\neq 0$.
	By Lemma \ref{lem:EquiRankSub}, we have
	\begin{align*}
		\max\set{\Rrankmax(W_i): i\in I} &\leq \Rrankmax(V), \\
		\min\set{\Rrankmin(W_i): i\in I} &\geq \Rrankmin(V).
	\end{align*}
	The converse of the first inequality holds since $\bigoplus_{i \in I} \rfilt{\Rrankmax(W_i)}{W_i^\infty}$ is dense in $V$.
	To show the converse of the second inequality, take a non-zero closed equi-rank subrepresentation $U\subset V$ with $\Rrank(U) = \Rrankmin(V)$.
	Then there is $i \in I$ such that the projection of $U$ to $W_i$ is non-zero.
	This implies $\Rrankmin(W_i)\leq \Rrank(U)$ by Lemma \ref{lem:EquiRankSub},
	and hence we have shown the assertion.
\end{proof}

\begin{proposition}\label{prop:CanonicalRankDecomposition}
	Let $V$ be a unitary representation of $G_\RR$.
	There exists a unique orthogonal direct sum decomposition $V = \bigoplus_{i=\Rrankmin(V)}^{\Rrankmax(V)} W_i$
	such that $\Rrank(W_i) = i$ or $W_i = 0$ for any $i$.
\end{proposition}

\begin{proof}
	For each $i$, set $W_i \coloneq (\rfilt{i-1}{V^\infty})^\perp \cap \overline{\rfilt{i}{V^\infty}}$.
	Then $\set{W_i}_i$ satisfies the desired condition.

	Let $V = \bigoplus_{i=\Rrankmin(V)}^{\Rrankmax(V)} W'_i$ be an orthogonal direct sum decomposition satisfying the conditions in the assertion.
	By Lemma \ref{lem:EquiRankSub}, the projection of $W'_i$ to $W_j$ is zero for any $i, j$ ($i < j$).
	The same is true by swapping $W_i$ for $W'_i$.
	Hence we have $\bigoplus_{i=\Rrankmin(V)}^k W_i = \bigoplus_{i=\Rrankmin(V)}^k W'_i$ for any $k$.
	By induction on $i$, we obtain $W_i = W'_i$.
	This shows the uniqueness.
\end{proof}

We shall state a relation between the rank and the discrete decomposability.
A unitary representation $V$ of $G_\RR$ is said to be discretely decomposable if $V$ is isomorphic to a direct sum of irreducible representations.
The following proposition is easy from definition.
This is one of the reasons why we define the rank.

\begin{proposition}
	Let $V$ be a unitary representation of $G_\RR$.
	Then $V$ is discretely decomposable if and only if $\Rrankmax(V) = 0$.
\end{proposition}

\begin{proof}
	It is clear that $V$ is discretely decomposable only if $\Rrankmax(V) = 0$.
	The converse follows from the fact that the number of isomorphism classes of irreducible unitary representations with a fixed infinitesimal character is finite.
\end{proof}

A similar result for $\lie{g}$-modules holds in several good categories, e.g.\ the category of $(\lie{g}, K)$-modules and the BGG category $\mathcal{O}$.
We shall state the result only for $(\lie{g}, K)$-modules.
For a $\lie{g}$-module $V$, we say that $V$ is discretely decomposable if
$V$ has a filtration $0 = V_0\subset V_1 \subset V_2 \subset \cdots$ such that $\bigcup_{i \in \NN} V_i = V$ 
and $V_i$ has finite length for any $i \in \NN$.
See \cite[Definition 1.1]{Ko98_discrete_decomposable_3}.
By definition, a $\lie{g}$-module $V$ is discretely decomposable only if $V$ has at most countable dimension.

\begin{proposition}
	Let $V$ be a $(\lie{g}, K)$-module of at most countable dimension.
	Then $V$ is discretely decomposable if and only if $\Rrankmax(V) = 0$.
\end{proposition}

\begin{proof}
	Any finitely generated $(\lie{g}, K)$-module with an infinitesimal character has finite length by Harish-Chandra's admissibility theorem \cite[Theorem 3.4.1]{Wa88_real_reductive_I}.
	The assertion follows from this fact.
\end{proof}

\subsection{Equi-rank theorem}

We shall prove the equi-rank property under some reasonable assumption in this subsection.
We consider the $\univcent{g}$-action on $\alg{A}$-modules.
Let $(\alg{A}, G)$ be a generalized pair with connected reductive algebraic group $G$.
Fix a Cartan subalgebra $\lie{h}\subset \lie{g}$ and let $W$ denote the Weyl group of $\lie{g}$.

In general, an $\alg{A}$-module may not be finitely generated as a $\lie{g}$-module.
To treat such a module, we consider existence of good $\lie{g}$-module quotients of an $\alg{A}$-module.

\begin{lemma}\label{lem:FiniteGenFiltration}
	Let $V$ be a non-zero finitely generated $\lie{g}$-module.
	Then there exists a $\lie{g}$-module filtration $0=V_0\subset V_1 \subset \cdots \subset V_r = V$ such that the action of $\univcent{g}/\Ann_{\univcent{g}}(V_i/V_{i-1})$ on $V_i/V_{i-1}$ is torsion-free for any $i \geq 1$.
\end{lemma}

\begin{proof}
	Take a maximal element $P$ of $\set{\Ann_{\univcent{g}}(v) : 0\neq v \in V}$.
	Then $P$ is a prime ideal such that $V_1 \coloneq V^P$ is a torsion-free $\univcent{g}/P$-module.
	Applying this argument for $V/V_{i}$ iteratively, we obtain an increasing sequence $0 = V_0 \subset V_1 \subset \cdots$ such that the action of $\univcent{g}/\Ann_{\univcent{g}}(V_i/V_{i-1})$ on $V_i/V_{i-1}$ is torsion-free for any $i \geq 1$.
	Since $V$ is noetherian, there is $r \in \NN$ such that $V_r = V$.
	This shows the lemma.
\end{proof}

\begin{lemma}\label{lem:FiniteGenInfChar}
	Let $V$ be a non-zero finitely generated $\lie{g}$-module and $\chi \in \lie{h}^*/W$.
	Assume that $U^{\Ker(\chi)} \neq 0$ for any non-zero submodule $U\subset V$.
	Then $V$ has the generalized infinitesimal character $\chi$.
\end{lemma}

\begin{proof}
	Let $X \in \Ker(\chi)$.
	We shall see that $X$ is nilpotent on $V$.
	Since $V$ is noetherian, there is $k \in \NN$ such that $V^{X^{k+1}} = V^{X^k}$.
	Assume that $X^k V\neq 0$.
	Since $X^k V$ is a non-zero $\lie{g}$-submodule of $V$, we have $0\neq (X^k V)^{\Ker(\chi)} \subset (X^k V)^{X}$ by assumption.
	This implies $V^{X^{k+1}} \neq V^{X^k}$ and this contradicts the choice of $k$.
	Hence we obtain $X^k V = 0$.

	Since the ideal $\Ker(\chi)\subset \univcent{g}$ is finitely generated,
	we can take $n \in \NN$ such that $\Ker(\chi)^n V = 0$.
	This shows the lemma.
\end{proof}

\begin{lemma}\label{lem:ExistenceQuotient}
	Let $V$ be a $\lie{g}$-module of at most countable dimension, and $P$ a prime ideal of $\univcent{g}$.
	Assume $V_P/ PV_P \neq 0$, where $V_P$ is the localization at $P$.
	Set
	\begin{align*}
		X\coloneq \set{\chi \in \Variety(P) : V/\Ker(\chi)V \neq 0}.
	\end{align*}
	Then $\Variety(P) - X$ is contained in an at most countable union of closed subvarieties strictly contained in $\Variety(P)$.
	If, moreover, $\Ann_{\univcent{g}}(V) = P$ and $V$ is torsion-free as a $\univcent{g}/P$-module,
	then $\bigcap_{\chi \in X} \Ker(\chi) V = 0$.
\end{lemma}

\begin{proof}
	Replacing $V$ with the image of $V\rightarrow V_P/PV_P$, we can assume that 
	$V$ is torsion-free as a $\univcent{g}/P$-module.
	Note that $V_P/PV_P$ is a vector space over the field $\univcent{g}_P/P \univcent{g}_P$.
	Since $V$ has at most countable dimension, we can take at most countable multiplicative subset $S\subset \univcent{g}/P$ such that
	$S^{-1}V$ is free over $S^{-1}(\univcent{g}/P)$.
	Hence we obtain $\Variety(P) - X \subset \bigcup_{f \in S} \Variety((f))$ and $\bigcap_{\chi \in X} \Ker(\chi) V = 0$.
\end{proof}

\begin{remark}\label{remark:ExistenceQuotient}
	The assumption that $V$ has at most countable dimension is necessary.
	Let $K$ denote the quotient field of $\univcent{g}/P$.
	Then $X$ in Lemma \ref{lem:ExistenceQuotient} is empty for the left $\lie{g}$-module $V = \univ{g}/P \univ{g}\otimes_{\univcent{g}/P}K$.
\end{remark}

\begin{lemma}\label{lem:FiniteGenDenseQuot}
	Let $V$ be a finitely generated $\lie{g}$-module.
	Set
	\begin{align*}
		X\coloneq \set{\chi \in \lie{h}^*/W : V/\Ker(\chi)V \neq 0}.
	\end{align*}
	Then $X$ is a Zariski dense subset of $\Variety(\Ann_{\univcent{g}}(V))$.
\end{lemma}

\begin{proof}
	First, we shall show
	\begin{align}
		X = \set{\chi \in \lie{h}^*/W : V \text{ has a subquotient with the infinitesimal character $\chi$}}. \nonumber \\
		\label{eqn:DenseSubquotient}
	\end{align}
	Let $V_1$ and $V_2$ be submodules of $V$ such that $V_2\subset V_1$ and $V_1 / V_2$ has an infinitesimal character $\chi$.
	Take a maximal submodule $W$ of $V$ satisfying $W \cap V_1 = V_2$.
	Then the intersection of any non-zero submodule of $V/W$ with $V_1/V_2$ is non-zero.
	By Lemma \ref{lem:FiniteGenInfChar}, $V/W$ has the generalized infinitesimal character $\chi$
	and hence $V$ has a quotient with the infinitesimal character $\chi$.
	This shows \eqref{eqn:DenseSubquotient}.

	Put $P\coloneq \Ann_{\univcent{g}}(V)$.
	By Lemma \ref{lem:FiniteGenFiltration} and \eqref{eqn:DenseSubquotient},
	it suffices to show the lemma in the case that $P$ is prime and $V$ is a torsion-free $\univcent{g}/P$-module.
	We have already shown this in Lemma \ref{lem:ExistenceQuotient}.
\end{proof}

Let $q\colon \lie{h}^*\rightarrow \lie{h}^*/W$ denote the quotient map.
The following lemma is a generalization of Kostant's theorem (Fact \ref{fact:Kostant}).
The lemma plays a central role in this paper.

\begin{lemma}\label{lem:TranslationSupport}
	Let $V$ be a finitely generated $\lie{g}$-module and $F$ a completely reducible finite-dimensional $\lie{g}$-module.
	Set $X\coloneq q^{-1}(\Variety(\Ann_{\univcent{g}}(V)))$.
	\begin{enumerate}
		\item $\Variety(\Ann_{\univcent{g}}(F\otimes V)) \subset \bigcup_{\mu \in R} q(X + \mu)$, where $R$ is the set of all $\lie{h}$-weights in $F$.
		\item $\Rrankmax(F\otimes V) = \Rrankmax(V)$, $\Rrankmin(F\otimes V) = \Rrankmin(V)$.
	\end{enumerate}
\end{lemma}

\begin{proof}
	Let $W$ be a quotient of $F\otimes V$ with an infinitesimal character $[\lambda]\in \lie{h}^*/W$.
	By $\Hom_{\lie{g}}(F\otimes V, W) \simeq \Hom_{\lie{g}}(V, F^*\otimes W)$, there exists a non-zero homomorphism $V\rightarrow F^*\otimes W$.
	By Fact \ref{fact:Kostant}, there exists $\mu \in R$ such that $V$ has a quotient with the infinitesimal character $[\lambda - \mu]$.
	This implies $[\lambda] \in \bigcup_{\mu \in R} q(X + \mu)$.
	This and Lemma \ref{lem:FiniteGenDenseQuot} show the first assertion.

	To show the second assertion, we can assume that $V$ is equi-rank by Proposition \ref{prop:RankFiltr}.
	By the first assertion and Proposition \ref{prop:RankFinitelyGen}, we have
	$\Rrankmax(F\otimes V) \leq \Rrank(V)$.
	Take an equi-rank submodule $W\subset F\otimes V$ with $\Rrank(W) = \Rrankmin(F\otimes V)$.
	By $\Hom_{\lie{g}}(W, F\otimes V) \simeq \Hom_{\lie{g}}(F^*\otimes W, V)$, there exists a non-zero homomorphism $F^*\otimes W \rightarrow V$.
	This shows
	\begin{align*}
		\Rrank(V) \leq \Rrankmax(F^*\otimes W) \leq \Rrank(W) = \Rrankmin(F\otimes V).
	\end{align*}
	Combining this with $\Rrankmax(F\otimes V) \leq \Rrank(V)$, we have shown the lemma.
\end{proof}

\begin{corollary}\label{cor:TranslationSupportNonFinGen}
	Let $V$ be a $\lie{g}$-module and $F$ a completely reducible finite-dimensional $\lie{g}$-module.
	\begin{enumerate}
		\item $\Rrankmax(F\otimes V) = \Rrankmax(V)$, $\Rrankmin(F\otimes V) = \Rrankmin(V)$.
		\item $\glrank(F\otimes V) = \glrank(V)$.
	\end{enumerate}
\end{corollary}

\begin{proof}
	The first assertion is clear from Lemma \ref{lem:TranslationSupport} since $V$ is a union of finitely generated submodules.
	Regard $V'\coloneq \univ{g}/\Ann_{\univ{g}}(V)$ as a $\lie{g}$-module via the left action.
	Then we have $\Ann_{\univ{g}}(F\otimes V) = \Ann_{\univ{g}}(F\otimes V')$ because
	they are the kernel of the composition $\univ{g}\xrightarrow{\Delta} \univ{g}\otimes \univ{g}\rightarrow \End_{\CC}(F)\otimes \univ{g}/\Ann_{\univ{g}}(V)$, where $\Delta$ is the coproduct.
	Since $V'$ is finitely generated, the second assertion follows from Proposition \ref{prop:RankFinitelyGen} and Lemma \ref{lem:TranslationSupport}.
\end{proof}

As a consequence of Lemma \ref{lem:TranslationSupport}, we obtain the following equi-rank theorem.

\begin{theorem}\label{thm:RankFiltration}
	Let $V$ be an $\alg{A}$-module.
	Then the rank filtration of $V$ defined in Definition \ref{def:Rank} is a filtration as an $\alg{A}$-module.
	If, moreover, $V$ is irreducible, then $V$ is equi-rank.
\end{theorem}

\begin{proof}
	Fix $k \in \NN$.
	Let $v \in \rfilt{k}{V}$ and $X \in \alg{A}$.
	We shall show $X\cdot v \in \rfilt{k}{V}$.
	Recall that the adjoint action of $\lie{g}$ on $\alg{A}$ is locally finite and completely reducible.
	Hence there exists a finite-dimensional $\lie{g}$-submodule $F\subset \alg{A}$ such that $X \in F$.
	Let $\varphi \colon F\otimes V\rightarrow V$ denote the multiplication map.
	Then we have $X\cdot v \in F\cdot \univ{g}v = \varphi(F\otimes \univ{g}v)$.
	By Lemma \ref{lem:TranslationSupport}, we obtain
	\begin{align*}
		k = \glrank(\univ{g}v) &= \glrank(F\otimes \univ{g}v) \\
		&\geq \glrank(F\cdot \univ{g}v) \geq \glrank(\univ{g}Xv).
	\end{align*}
	This implies $Xv \in \rfilt{k}{V}$.
\end{proof}

\begin{corollary}\label{cor:GenPrincipalComponent}
	Let $V$ be an irreducible $\alg{A}$-module.
	The sum $\sum_{P \in \Ass_{\univcent{g}}(V)} V^P$ is a direct sum.
\end{corollary}

\begin{proof}
	Let $P_1, \ldots, P_r \in \Ass_{\univcent{g}}(V)$ such that $P_i\neq P_j$ ($i\neq j$).
	Since $V$ is equi-rank, $V^{(P_1 \cap P_2 \cap \cdots \cap P_i) + P_{i+1}} = 0$
	for any $1\leq i < r$.
	This shows
	\begin{align*}
		\left(\sum_{j=1}^i V^{P_j}\right) \cap V^{P_{i+1}} \subset V^{(P_1 \cap P_2 \cap \cdots \cap P_i) + P_{i+1}} = 0 \quad (1\leq \forall i < r)
	\end{align*}
	and hence $\sum_{P \in \Ass_{\univcent{g}}(V)} V^P$ is a direct sum.
\end{proof}

\subsection{Cartan subalgebra}

In the previous subsection, we have shown that any irreducible $\alg{A}$-module is equi-rank as a $\lie{g}$-module.
In this subsection, we shall consider the support $\Variety(\Ann_{\univcent{g}}(v))$ and define `Cartan subalgebra' for $\alg{A}$-modules.

Let $(\alg{A}, G)$ be a generalized pair with connected reductive algebraic group $G$.
Assume that $\alg{A}$ has at most countable dimension.
Then any irreducible $\alg{A}$-module has at most countable dimension.
This assumption is used only for the existence of enough many $\lie{g}$-module quotients with infinitesimal characters.
Fix a Cartan subalgebra $\lie{h}$ of $\lie{g}$ and write $W$ for the Weyl group of $\lie{g}$.
Let $q \colon \lie{h}^* \rightarrow \lie{h}^*/W$ denote the quotient map.

\begin{lemma}\label{lem:AnnTranslation}
	Let $V$ be a $\lie{g}$-module and let $F$ be a completely reducible finite-dimensional $\lie{g}$-module.
	Assume that $P\coloneq \Ann_{\univcent{g}}(V)$ is a prime ideal and $V$ is torsion-free as a $\univcent{g}/P$-module.
	Fix an irreducible component $A$ of $q^{-1}(\Variety(P))$.
	Then there exist prime ideals $Q_1, \ldots, Q_r\subset \univcent{g}$ satisfying
	\begin{enumerate}
		\item $Q_1^{|W|}\cap \cdots \cap Q_r^{|W|}(F\otimes V) = 0$,
		\item for each $Q_i$, there exists an $\lie{h}$-weight $\mu$ of $F$ such that
		$\Variety(Q_i) = q(A + \mu)$.
	\end{enumerate}
\end{lemma}

\begin{proof}
	Put $X\coloneq \set{\lambda \in A : V/\Ker(\chi_\lambda)V \neq 0}$.
	By Lemma \ref{lem:ExistenceQuotient}, the natural homomorphism $V\rightarrow \prod_{\lambda \in X} V/\Ker(\chi)V$ is injective.
	Hence we have an injection
	\begin{align*}
		F\otimes V\hookrightarrow \prod_{\lambda \in X} F\otimes V/\Ker(\chi_\lambda)V.
	\end{align*}
	Let $R$ denote the set of all $\lie{h}$-weights in $F$.
	By Fact \ref{fact:Kostant}, we obtain
	\begin{align*}
		\Ann_{\univcent{g}}(F\otimes V) \supset \bigcap_{\lambda \in X, \mu \in R} \Ker(\chi_{\lambda + \mu})^{|W|}\eqcolon Q.
	\end{align*}
	Since $X$ is dense in $A$, we have
	\begin{align*}
		\Variety(Q) = \bigcup_{\mu \in R} q(\overline{X} + \mu) = \bigcup_{\mu \in R} q(A + \mu).
	\end{align*}
	Hence the minimal prime ideals containing $Q$ satisfy the desired conditions.
\end{proof}

\begin{lemma}\label{lem:BoundAnn}
	Let $V$ be an irreducible $\alg{A}$-module.
	Then for any $0\neq v \in V$, there exist $P_1, \ldots, P_r \in \Ass_{\univcent{g}}(V)$ such that
	\begin{align*}
		P_1^{|W|}\cap \cdots \cap P_r^{|W|}\subset \Ann_{\univcent{g}}(v) \subset P_1 \cap \cdots \cap P_r.
	\end{align*}
\end{lemma}

\begin{proof}
	Let $0\neq v \in V$.
	Fix $P \in \Ass_{\univcent{g}}(V)$.
	By Theorem \ref{thm:RankFiltration}, $V^P$ is torsion-free as a $\univcent{g}/P$-module.
	Since $V$ is an irreducible $\alg{A}$-module, we can take a finite-dimensional $\lie{g}$-submodule $F\subset \alg{A}$ such that $v \in F\cdot V^P$.
	Take prime ideals $Q_1, \ldots, Q_s\subset \univcent{g}$ as in Lemma \ref{lem:AnnTranslation}
	for $F\otimes V^P$.
	Then we have $\Ann_{\univcent{g}}(v) \supset \bigcap_{i} Q_i^{|W|}$ and, by Theorem \ref{thm:RankFiltration},
	\begin{align*}
		\dim(\Variety(\Ann_{\univcent{g}}(v))) = \dim(\Variety(P)) = \dim(\Variety(Q_i)) 
	\end{align*}
	for any $i$.
	Remark that all the irreducible components of $\Variety(\Ann_{\univcent{g}}(v))$ have the same dimension by Theorem \ref{thm:RankFiltration}.
	Hence there exists a subset $S\subset \set{1, \ldots, s}$ such that
	$\Variety(\Ann_{\univcent{g}}(v)) = \bigcup_{i \in S} \Variety(Q_i)$.
	This shows the assertion.
\end{proof}

\begin{lemma}\label{lem:IrrExistenceQuot}
	Let $V$ be an irreducible $\alg{A}$-module and $P \in \Ass_{\univcent{g}}(V)$.
	Then $V_P/PV_P \neq 0$ holds.
\end{lemma}

\begin{proof}
	From $P \in \Ass_{\univcent{g}}(V)$, we have $V_P\neq 0$.
	Put $\alg{Z}\coloneq \univcent{g}_P$.
	By Lemma \ref{lem:BoundAnn}, $P^{|W|}\alg{Z}\subset \Ann_{\alg{Z}}(v) \subset P\alg{Z}$ holds
	for any $v \in V_P$.
	Hence we can take $k \in \NN$ such that $P^{k+1}V_P = 0$.
	By $V_P \neq 0$, this implies $V_P \neq PV_P$.
\end{proof}

\begin{theorem}\label{thm:Affinity}
	Let $V$ be an irreducible $\alg{A}$-module.
	Let $R$ denote the set of all $\lie{h}$-weights in $\alg{A}$.
	Then there exist a subspace $\lie{a}^*\subset \lie{h}^*$ and $\lambda_0 \in \lie{h}^*$ satisfying the following properties.
	\begin{enumerate}
		\item $\lie{a}^*$ is spanned by some elements of the form $\mu - \mu'$ ($\mu, \mu'\in R$).
		\item For any $0\neq v \in V$, there exist finite elements $\alpha_1, \ldots, \alpha_r \in R$ such that
		\begin{align*}
			\Variety(\Ann_{\univcent{g}}(v)) = q\left(\bigcup_{i} (\lie{a}^* + \lambda_0 + \alpha_i)\right).
		\end{align*}
	\end{enumerate}
\end{theorem}

\begin{remark}\label{rmk:Affinity}
	In Theorem \ref{thm:Affinity}, the property that $V$ has at most countable dimension is necessary.
	Let $P$ be a prime ideal of $\univcent{g}$ such that $q^{-1}(\Variety(P))$ is not a finite union of affine subspaces in $\lie{h}^*$.
	Let $K$ denote the quotient field of $\univcent{g}/P$.
	For $\alg{A} = \univ{g}/P\univ{g}\otimes_{\univcent{g}/P} K$, any irreducible $\alg{A}$-module
	does not satisfy the conclusion of Theorem \ref{thm:Affinity}.
\end{remark}

\begin{proof}
	We shall construct $\lie{a}^*$.
	Fix $P \in \Ass_{\univcent{g}}(V)$
	and an irreducible component $A$ of $q^{-1}(\Variety(P))$.
	Then we have $q(A) = W\cdot A = \Variety(P)$.
	Put $X\coloneq \set{\lambda \in A : V/\Ker(\chi_{\lambda})V \neq 0}$.
	By Lemmas \ref{lem:ExistenceQuotient} and \ref{lem:IrrExistenceQuot}, $A - X$ is contained in an at most countable union of closed subvarieties strictly contained in $A$.
	Since $R$ is at most countable, we can take an element $\lambda \in X$ such that
	$\lambda \not \in w(A+\mu)$ for any $\mu \in R$ and $w \in W$ with $w(A + \mu) \neq A$.
	In other words, if $\lambda \in w(A + \mu)$ for some $\mu \in R$ and $w \in W$, we have $w(A + \mu) = A$.
	Let $\varphi \colon V\rightarrow V/\Ker(\chi_\lambda)V$ denote the natural surjection.

	For a finite-dimensional $\lie{g}$-submodule $F$ of $\alg{A}$, we consider the following composition map
	\begin{align*}
		\varphi_F \colon V^P \simeq (F^*\otimes F)^G \otimes V^P \hookrightarrow F^*\otimes (F\otimes V)
		\xrightarrow{\id\otimes (\varphi\circ m)} F^*\otimes V/\Ker(\lambda)V,
	\end{align*}
	where $m\colon F\otimes V\rightarrow V$ is the multiplication map.
	We shall show $\bigcap_{F\subset \alg{A}} \Ker(\varphi_F) = 0$.

	Let $0\neq v \in V^P$.
	Since $V$ is an irreducible $\alg{A}$-module, we can take a finite-dimensional $\lie{g}$-submodule $F'$ of $\alg{A}$ such that $\varphi(F'\cdot v) \neq 0$.
	Then we have $\varphi_{F'}(v) \neq 0$.
	This shows $\bigcap_{F\subset \alg{A}} \Ker(\varphi_F) = 0$ and hence
	$\prod_{F\subset \alg{A}} \varphi_F \colon V^P\rightarrow \prod_{F\subset \alg{A}} F^*\otimes V/\Ker(\chi_\lambda)V$ is injective.

	From this and Fact \ref{fact:Kostant}, there exists a subset $R'\subset R$ such that
	\begin{align*}
		\bigcap_{\mu \in R'} \Ker(\chi_{\lambda - \mu})^{|W|}
		\subset P \subset \bigcap_{\mu \in R'} \Ker(\chi_{\lambda - \mu}).
	\end{align*}
	This implies $\Variety(P) = q(\overline{\lambda - R'})$.
	Hence there exist a subset $R''\subset R'$ and $w \in W$ such that $w(\overline{\lambda - R''}) = A$.
	By the choice of $\lambda$, we have $w(A - \mu) = A = w(A - \mu')$ for any $\mu, \mu' \in R''$.
	Put $\lie{a}^* \coloneq \spn{\mu - \mu' : \mu,\mu' \in R''}$.
	Then $A$ is closed under the translation by any element of $\lie{a}^*$.
	Taking $\mu' \in R''$, we have
	\begin{align*}
		\dim(A) \geq \dim(\lie{a}^*) \geq \dim(\overline{\mu' - R''}) = \dim(\overline{\lambda - R''})
		= \dim(A).
	\end{align*}
	Since $A$ is irreducible, we can take $\lambda_0 \in A$ such that $A = \lie{a}^* + \lambda_0$,
	and hence $\Variety(P) = q(\lie{a}^* + \lambda_0)$.

	The second assertion follows from Lemma \ref{lem:AnnTranslation} and Theorem \ref{thm:RankFiltration}.
	See also the proof of Lemma \ref{lem:BoundAnn}.
\end{proof}

We shall give a similar result as Theorem \ref{thm:Affinity} for $\Variety(\Ann_{\univcent{g}}(V))$ assuming a strong condition.
The condition is fulfilled in some practical cases in the branching problem of reductive Lie groups.
See Theorem \ref{thm:AffinityRealSpherical}.

\begin{theorem}\label{thm:AffinityAnn}
	Let $V$ be an irreducible $\alg{A}$-module.
	Suppose that there exists a faithful irreducible $\alg{A}/\Ann_{\univcent{g}}(V)$-module $V'$ such that
	$V'|_{\lie{g}}$ is finitely generated.
	Let $R$ denote the set of all $\lie{h}$-weights in $\alg{A}$.
	Then there exists a subspace $\lie{c}^*\subset \lie{h}^*$ satisfying the following properties.
	\begin{enumerate}
		\item $\lie{c}^*$ is spanned by some elements of the form $\mu - \mu'$ ($\mu, \mu'\in R$).
		\item Any irreducible component of $\Variety(\Ann_{\univcent{g}}(V))$ is a translation of $\lie{c}^*$.
	\end{enumerate}
	Moreover, $\lie{c}^*$ can be taken to contain $\lie{a}^*$ in Theorem \ref{thm:Affinity}.
\end{theorem}

\begin{proof}
	Take a finite generating set $\set{v_1, \ldots, v_r}$ of $V'|_{\lie{g}}$.
	Then, by Proposition \ref{prop:RankFinitelyGen}, we have $\Ann_{\univcent{g}}(V) = \bigcap_{i}\Ann_{\univcent{g}}(v_i)$.
	Hence the assertion follows from Theorem \ref{thm:Affinity}.
\end{proof}

\begin{remark}
	If $\Ann_{\alg{A}}(V)$ is completely prime, $\Variety(\Ann_{\univcent{g}}(V))$ is irreducible.
	In general, $\Variety(\Ann_{\univcent{g}}(V))$ may not be irreducible.
	For example, if $V$ is finite dimensional and $V|_{\lie{g}}$ contains at least two non-isomorphic irreducible submodules, then $\Variety(\Ann_{\univcent{g}}(V))$ is not irreducible.
\end{remark}

Motivated by Theorem \ref{thm:Affinity}, we shall define a Cartan subalgebra for $V$.

\begin{definition}\label{def:Cartan}
	Let $V$ be a non-zero $\lie{g}$-module.
	We say that $V$ has a small Cartan subalgebra if there exists a subspace $\lie{a}^*\subset \lie{t}^*$ such that
	any $\Variety(P)$ ($P \in \Ass_{\univcent{g}}(V)$) is of the form $q(\lie{a}^* + \mu)$ ($\mu \in \lie{t}^*$).
	We call the algebra $\lie{h}/\bigcap_{\lambda \in \lie{a}^*} \Ker(\lambda)$ a \define{small Cartan subalgebra} for $V$.

	We say that $V$ has a Cartan subalgebra if there exists a subspace $\lie{c}^*\subset \lie{t}^*$ such that
	any irreducible component of $\Variety(\Ann_{\univcent{g}}(V))$ is of the form $q(\lie{c}^* + \mu)$ ($\mu \in \lie{t}^*$).
	We call the algebra $\lie{h}/\bigcap_{\lambda \in \lie{c}^*} \Ker(\lambda)$ a \define{Cartan subalgebra} for $V$.
\end{definition}

When we fix a $G$-invariant non-degenerate symmetric form on $\lie{g}$,
we use the same terminology for the corresponding subalgebras of $\lie{h}$ by the form.
Let us consider typical examples related to homogeneous spaces.
For an affine algebraic variety $X$, we denote by $\rring{X}$ the coordinate ring
and by $\ntDalg{X}$ the algebra of differential operators on $X$.
Fix a Borel subgroup $B$ of $G$.

\begin{example}
	Let $K$ be a reductive subgroup of $G$.
	Then $G/K$ is an affine variety and $\ntDalg{G/K}$ acts on $\rring{G/K}$ irreducibly.
	Since $\rring{G/K}$ is a direct sum of irreducible $G$-modules,
	we have $\Rrank(\rring{G/K}) = 0$, and hence $0 = \lie{h}/\lie{h}$ is the small Cartan subalgebra for $\rring{G/K}$.
	$\lie{c}^*$ can be taken as the space spanned by all the highest weights in $\rring{G/K}$.
	In particular, $\Variety(\Ann_{\univcent{g}}(\rring{G/K}))$ is irreducible.
	See \cite[Theorem 4.7 and Chapter 4]{Ti11} for the notion of Cartan subalgebras and Weyl groups of $G$-varieties.

	Remark that $\Ann_{\univcent{g}}(\rring{G/H})$ coincides with the kernel of the natural homomorphism $\univcent{g}\rightarrow \ntDalg{G/H}$.
	Hence the Cartan subalgebra is also that of any irreducible $\ntDalg{G/H}$-module.
\end{example}

\begin{example}
	Let $K$ be a symmetric subgroup of $G$.
	Then $V\coloneq \univ{g}\otimes_{\univ{k}}\CC$ is an irreducible $\ntDalg{G/K}$-module.
	It is well-known that $\ntDalg{G/K}^G$ is isomorphic to $\univ{a}^{W_{G/K}}$ for a maximal abelian subspace $\lie{a}\subset \lie{k}^\perp$ and $\ntDalg{G/K}^G$ is finitely generated as a $\univcent{g}$-module.
	Here $W_{G/K}$ is the little Weyl group for $(G, K)$.
	This implies that $\lie{a}$ is a Cartan subalgebra for $V$.
	Moreover, $\lie{a}$ is also a small Cartan subalgebra for $V$ from $\Ann_{\univcent{g}}(V) = \Ann_{\univcent{g}}(1\otimes 1)$ (see Proposition \ref{prop:RankFinitelyGen}).
\end{example}

At the last of this subsection, we give a remark on the case that $V$ has an invariant Hermitian inner product.
In the case, we can refine Lemma \ref{lem:BoundAnn} as follows.
Let $\overline{(\cdot)}\colon \lie{g}\rightarrow \lie{g}$ be a complex conjugate.
Then $-\overline{(\cdot)}$ extends to an anti-involution $(\cdot)^*\colon \univ{g}\rightarrow \univ{g}$.
Write $\lie{g}_\RR$ for the real form of $\lie{g}$.

\begin{proposition}\label{prop:UnitaryRadicalIdeal}
	Let $V$ be an $\alg{A}$-module with a $\lie{g}_\RR$-invariant Hermitian inner product $\langle\cdot, \cdot \rangle$.
	Then $\Ann_{\univcent{g}}(v)$ is $*$-stable radical ideal for any $0\neq v \in V$,
	and so is $\Ann_{\univcent{g}}(V)$.
\end{proposition}

\begin{proof}
	Let $0\neq v \in V$ and set $I\coloneq \Ann_{\univcent{g}}(v)$.
	First we shall show $I^* = I$.
	Let $X \in I$.
	Then we have $0 = \langle Xv, Xv \rangle = \langle X^*v, X^*v\rangle$.
	This shows $X^* \in I$ and hence $I^* = I$.
	Note that $I^* = I$ implies $\sqrt{I}^* = \sqrt{I}$.

	Let $X \in \sqrt{I}$ and we shall show $X \in I$.
	Take $k \geq 1$ such that $X^{2^k} \in I$.
	Suppose that $X^* = X$.
	Then we have
	\begin{align*}
		\langle X^{2^{k-1}}v, X^{2^{k-1}}v\rangle = \langle X^{2^k}v, v\rangle = 0.
	\end{align*}
	This implies $X^{2^{k-1}} \in I$.
	By induction on $k$, we obtain $X \in I$.

	Similarly, if $X^* = -X$, then $X \in I$ holds.
	For the general case, since $X$ decomposes as $X = (X + X^*) /2 + (X - X^*)/2$,
	we obtain $X \in I$, and hence $I = \sqrt{I}$.

	By the same argument, $\Ann_{\univcent{g}}(V)$ is also $*$-stable and a radical ideal.
	We have proved the assertion.
\end{proof}

\subsection{Abelian case}\label{subsection:AbelianGeneral}

We shall consider the case that $G$ is abelian.
In this case, the Weyl group $W$ is trivial, so several complexities in the non-abelian case vanish.
The results in this subsection is useful to study the space of $\lie{u}$-coinvariants
and actions of non-compact centers of reductive Lie groups.

Let $T$ be a complex torus and $(\alg{A}, T)$ a generalized pair.
Suppose that $\alg{A}$ has at most countable dimension.
For each $\lambda \in \lie{t}^*$, we denote by $\alg{A}^\lambda$ the weight space of weight $\lambda$.
The following theorem is a generalization of weight space decomposition.

\begin{theorem}\label{thm:DirectSumAbelian}
	Let $V$ be an irreducible $\alg{A}$-module.
	Then $V$ has the direct sum decomposition $V = \bigoplus_{P \in \Ass_{\univ{t}}(V)} V^P$.
\end{theorem}

\begin{proof}
	In Corollary \ref{cor:GenPrincipalComponent}, we have shown that the sum $\sum_{P \in \Ass_{\univ{t}}(V)} V^P$ is a direct sum.
	It is enough to show $V = \sum_{P \in \Ass_{\univ{t}}(V)} V^P$.
	
	Fix $P \in \Ass_{\univ{t}}(V)$.
	By Theorem \ref{thm:Affinity}, there exist a subspace $\lie{a}^* \subset \lie{t}^*$ and $\lambda_0 \in \lie{t}^*$ such that $\Variety(P) = \lie{a}^* + \lambda_0$.
	Let $R$ denote the set of all $\lie{t}$-weights in $\alg{A}$.
	Since $V$ is irreducible, we have $V = \sum_{\lambda \in R} \alg{A}^\lambda V^P$.
	For any $\lambda \in R$, $\alg{A}^\lambda V^P$ is contained in $V^Q$ for some $Q\in \Ass_{\univ{t}}(V)$ by Lemma \ref{lem:AnnTranslation}.
	Therefore we have shown $V = \sum_{P \in \Ass_{\univ{t}}(V)} V^P$.
\end{proof}

By the proof of Theorem \ref{thm:DirectSumAbelian}, the following corollary follows.

\begin{corollary}\label{cor:AbelianGen}
	Let $V$ be an irreducible $\alg{A}$-module.
	Then there exists a connected closed subgroup $M\subset T$ satisfying the following properties.
	\begin{enumerate}
		\item $\lie{t}/\lie{m}$ is the small Cartan subalgebra for $V$.
		\item The action of $\lie{m}$ on $V$ is locally finite and completely reducible.
		\item For each $P \in \Ann_{\univ{t}}(V)$, $V^P$ is an irreducible $\alg{A}^{M}$-module 
		and a weight space of $\lie{m}$.
		\item For each $\lie{m}$-weight $\lambda, \mu$ in $V$, $\alg{A}^{\lambda - \mu}\cdot V^\mu = V^\lambda$ holds.
	\end{enumerate}
	Here $V^\lambda$ is the weight space of weight $\lambda$.
\end{corollary}

\begin{proof}
	Take a small Cartan subalgebra $\lie{a}=\lie{t}/\lie{m}$ for $V|_{\lie{g}}$.
	Then $\lie{a}^*$ is spanned by some characters that lift to characters of $T$ by Theorem \ref{thm:Affinity}.
	Hence $\lie{m}= \bigcap_{\lambda \in \lie{a}^*} \Ker(\lambda)$ is a Lie algebra of a connected closed subgroup of $T$, which is $M$ in the assertion.
	Then the theorem follows from Theorem \ref{thm:DirectSumAbelian}.
\end{proof}

\begin{remark}
	To show Corollary \ref{cor:AbelianGen}, it is enough to assume $\End_{\alg{A}}(V) = \CC$ instead of the countability of $\alg{A}$ (or $V$).
	We give a simple proof of Corollary \ref{cor:AbelianGen} in \cite{Ki24-2}.
\end{remark}

Corollary \ref{cor:AbelianGen} asserts that the action on $V$ determines a compact part and a split part of $\lie{t}$.
Conversely, if we fix a split real form of $T$, the $\lie{t}$-action is strongly restricted as follows.
Let $\lie{t}_\RR$ be a real form of $\lie{t}$ such that any non-zero $\lie{t}_\RR$-weight in $\alg{A}$
is not unitary.
Write $\overline{(\cdot)}\colon \lie{t}\rightarrow \lie{t}$ for the complex conjugate with respect to $\lie{t}_\RR$ and extend $-\overline{(\cdot)}$ to an anti-linear involution $(\cdot)^*\colon \univ{t}\rightarrow \univ{t}$.

\begin{theorem}\label{thm:AbelinUnitary}
	Let $V$ be an irreducible $\alg{A}$-module with a $\lie{t}_\RR$-invariant Hermitian inner product $\langle \cdot, \cdot \rangle$.
	Then we have $\Ass_{\univ{t}}(V) = \set{\Ann_{\univ{t}}(V)}$.
\end{theorem}

\begin{proof}
	Let $P \in \Ass_{\univ{t}}(V)$ and take a subspace $\lie{a}^* \subset \lie{t}^*$ and $\lambda_0 \in \lie{t}^*$ such that $\Variety(P) = \lie{a}^* + \lambda_0$ as in Theorem \ref{thm:Affinity}.
	By Proposition \ref{prop:UnitaryRadicalIdeal}, $P$ is $*$-stable,
	and hence $\lie{a}^* + \lambda_0$ is stable under $-\overline{(\cdot)}$.
	This implies that $\lie{a}^*$ is stable under the complex conjugate, and $\lambda_0 + \overline{\lambda_0} \in \lie{a}^*$.

	Let $Q \in \Ass_{\univ{t}}(V)$.
	Then there exists a $\lie{t}$-weight $\lambda$ in $\alg{A}$ such that $\Variety(Q) = \Variety(P) + \lambda = \lie{a}^* + \lambda_0 + \lambda$.
	By the same argument in the above, $\overline{\lambda} = -\lambda$ holds.
	This implies that $\lambda$ is a unitary character of $\lie{t}_\RR$.
	By the assumption on $\lie{t}$-weights in $\alg{A}$, we obtain $\lambda = 0$ and hence $Q = P$.
	We have shown the theorem.
\end{proof}

\subsection{\texorpdfstring{$\lie{n}$}{n}-coinvariant}

As a supplement, we shall give a relation between the Cartan subalgebra for a $\lie{g}$-module $V$ and that for the $\lie{l}$-module $V/\lie{n}V$.

Let $(\alg{A}, G)$ be a generalized pair with connected reductive algebraic group $G$.
Suppose that $\alg{A}$ has at most countable dimension.
Fix a Borel subalgebra $B = TU$ of $G$ with a maximal torus $T$ and unipotent radical $U$.
Let $R$ be a standard parabolic subgroup of $G$ with standard Levi decomposition $R = LN$.
Write $W$ (resp.\ $W_L$) for the Weyl group of $G$ (resp.\ $L$) and $q_L\colon \lie{t}^*/W_L \rightarrow \lie{t}^*/W$ for the quotient map.
We denote by $\rho_{\lie{n}}$ half the sum of all roots in $\lie{n}$.

We study the space $V/\lie{n}V$ of $\lie{n}$-coinvariants in an $\alg{A}$-module $V$.
To guarantee $V/\lie{n}V \neq 0$, we restrict the problem to Harish-Chandra modules.
We need the following property.
Let $\theta$ be an involution of $G$ and $K$ a (connected and finite) covering of $(G^\theta)_0$.
Suppose $\lie{k} + \lie{b} = \lie{g}$.

\begin{lemma}\label{lem:ExistenceIrrQuot}
	Let $V$ be a non-zero $(\lie{g}, K)$-module with an infinitesimal character $\chi \in \lie{t}^*/W$.
	Then $V$ has an irreducible quotient and $V/\lie{n}V$ is a non-zero $\lie{l}$-module whose irreducible quotient has an infinitesimal character in $q_L^{-1}(\chi) + \rho_{\lie{n}}$.
\end{lemma}

\begin{proof}
	It is well-known that $V/\lie{u}V\neq 0$ (and hence $V/\lie{n}V\neq 0$) if $V$ is irreducible (see \cite[Theorem 3.8.3]{Wa88_real_reductive_I}),
	so it is enough to show that $V$ has an irreducible quotient.
	Since the number of equivalence classes of irreducible $(\lie{g}, K)$-modules with the infinitesimal character $\chi$,
	there exist finitely many $K$-types $\tau_1, \ldots, \tau_r$ such that
	$V$ is generated by $\bigoplus_{i} V(\tau_i)$.
	Here $V(\tau_i)$ is the isotypic component in $V$ of the $K$-type $\tau_i$.

	Let $1\leq i \leq r$ and $F$ be an irreducible submodule of $V(\tau_i)|_K$.
	Then the multiplication map $\univ{g}\otimes F\rightarrow V$ factors through $\univ{g}\otimes_{\univcent{g}\univ{k}}F$.
	Since $\univ{g}\otimes_{\univcent{g}\univ{k}}F$ is a finitely generated $(\lie{g}, K)$-module and has infinitesimal character, it has finite length by the Harish-Chandra admissibility theorem and \cite[Theorem 4.2.1]{Wa88_real_reductive_I}.
	Hence the socle filtration of $\univ{g}\otimes_{\univcent{g}\univ{k}}F$ is finite, and the length of the filtration depends only on $i$.
	Since $V$ is the union of all submodules of the form $\univ{g}\cdot F$, the socle filtration of $V$ is also finite,
	and hence $V$ has an irreducible quotient.
\end{proof}

Let $\overline{N}$ denote the unipotent radical of the opposite parabolic subgroup of $R$.

\begin{proposition}\label{prop:Ucoinv}
	Let $V$ be an irreducible $(\alg{A}, K)$-module.
	Set $n\coloneq\Rrank(V|_{\lie{g}})$.
	Then $\Rrankmax(V/\lie{n}V) = n$ holds.
	Moreover, for any $P \in \Ass_{\univcent{l}}(V/\lie{n}V)$ with $\dim(\Variety(P)) = n$,
	$\Variety(P)$ is the image of a translation of the dual of a small Cartan subalgebra of $V|_{\lie{g}}$ by the quotient $\lie{t}^*\rightarrow \lie{t}^*/W_L$.
\end{proposition}

\begin{proof}
	Write $p\colon V\rightarrow V/\lie{n}V$ for the quotient map.
	Then we have $\Ann_{\univcent{g}}(v)\subset \Ann_{\univcent{l}}(p(v))$ for any $v \in V$, where $\univcent{g}$ is identified with a subalgebra of $\univcent{l}$ via the quotient $\univ{g}\rightarrow \univ{g}/(\lie{n}\univ{g} + \univ{g}\lie{\overline{n}})\simeq \univ{l}$.
	This implies $\Rrankmax(V/\lie{n}V) \leq n$ and
	\begin{align}
		\bigcup_{P \in \Ass_{\univcent{g}}(V)} \Variety(P) \supset \bigcup_{Q \in \Ass_{\univ{t}}(V/\lie{n}V)} q_L(\Variety(Q)- \rho_{\lie{n}}).
		\label{eqn:Ucoinv}
	\end{align}

	Let $P \in \Ass_{\univcent{g}}(V)$ and put $X\coloneq \set{\chi \in \Variety(P): V/\Ker(\chi)V\neq 0}$.
	By Lemma \ref{lem:ExistenceIrrQuot}, we have
	\begin{align*}
		X \subset \bigcup_{Q \in \Ass_{\univcent{l}}(V/\lie{n}V)} q_L(\Variety(Q) - \rho_{\lie{n}}),
	\end{align*}
	and by Lemmas \ref{lem:ExistenceQuotient} and \ref{lem:IrrExistenceQuot}, $\Variety(P) - X$ is contained in an at most countable union of closed subvarieties strictly contained in $\Variety(P)$.
	Hence we have
	\begin{align*}
		\Variety(P) = (\Variety(P) - X) \cup \bigcup_{Q \in \Ass_{\univcent{l}}(V/\lie{n}V)} (q_L(\Variety(Q) - \rho_{\lie{n}}) \cap \Variety(P)).
	\end{align*}
	Note that $\Ass_{\univcent{l}}(V/\lie{n}V)$ is an at most countable set since $V$ has at most countable dimension.
	Since $\Variety(P)$ is irreducible and $\dim(\Variety(Q)) \leq n = \dim(\Variety(P))$ for any $Q \in \Ass_{\univcent{l}}(V/\lie{n}V)$, there exists $Q' \in \Ass_{\univcent{l}}(V/\lie{n}V)$ such that $q_L(\Variety(Q') - \rho_{\lie{n}}) = \Variety(P)$.
	This shows the desired equality $\Rrankmax(V/\lie{n}V) = n$.

	Let $Q \in \Ass_{\univcent{l}}(V/\lie{n}V)$ with $\dim(\Variety(Q)) = n$.
	By \eqref{eqn:Ucoinv} and the same argument as in the above, there exists $P \in \Ass_{\univcent{g}}(V)$ such that $\Variety(P) = q_L(\Variety(Q) - \rho_{\lie{n}})$.
	Since $\Variety(Q)$ is irreducible, $\Variety(Q)$ is an irreducible component of $q_L^{-1}(\Variety(P)) + \rho_{\lie{n}}$.
	This shows the last assertion.
\end{proof}

In the proof of Proposition \ref{prop:Ucoinv}, we have shown the following result.

\begin{corollary}
	Let $V$ be an irreducible $(\alg{A}, K)$-module.
	Set $n\coloneq \Rrank(V|_{\lie{g}})$.
	Then we have
	\begin{align*}
		\bigcup_{P \in \Ass_{\univcent{g}}(V)} \Variety(P) &= \bigcup_{\substack{Q \in \Ass_{\univcent{l}}(V/\lie{n}V) \\ \dim(\Variety(Q)) = n}} q_L(\Variety(Q) - \rho_{\lie{n}})
	\end{align*}
	and
	\begin{align*}
		\Variety(\Ann_{\univcent{g}}(V)) &= q_L(\Variety(\Ann_{\univcent{l}}(V/\lie{n}V)) - \rho_{\lie{n}}) \\
		&= q_L(\Variety(\Ann_{\univcent{l}}(\rfilt{n}{V/\lie{n}V}/\rfilt{n-1}{V/\lie{n}V})) - \rho_{\lie{n}}).
	\end{align*}
\end{corollary}

%% file: branching.tex
\section{Branching problem}\label{section:Branching}

In the previous section, we introduce the notions of ranks and Cartan subalgebras for $\lie{g}$-modules.
In this section, we consider the branching problem of reductive Lie groups.

\subsection{Hilbert representation}\label{subsection:GKBasic}

Let $\widetilde{G}_\RR$ be a connected semisimple Lie group with finite center and $G_\RR$ a connected reductive subgroup of $\widetilde{G}_\RR$.
Fix a Cartan involution $\theta$ of $\widetilde{G}_\RR$ that stabilizes $G_\RR$
and set $\widetilde{K}_\RR \coloneq \widetilde{G}_\RR^\theta$ and $K_\RR \coloneq G_\RR \cap \widetilde{K}_\RR$.
Write $\widetilde{K}$ (resp.\ $K$) for the complexification of $\widetilde{K}_\RR$ (resp.\ $K_\RR$).

Write $G$ for the Zariski closure of $\Ad(G_\RR)$ in $\Aut(\widetilde{\lie{g}})$.
Suppose that $\Lie(G)$ is the complexification of $\lie{g}_\RR$.
This implies that $\lie{g}$ is algebraic in $\widetilde{\lie{g}}$.
Then $(\univ{\widetilde{g}}, G)$ is a generalized pair.
Although many results hold without the algebraicity, we assume this for simplicity.

For an admissible Hilbert representation $\hrep{H}$ of $\widetilde{G}_\RR$, we denote by $\hrep{H}_{\widetilde{K}}$ (resp.\ $\hrep{H}^\infty$, $\hrep{H}^\omega$) the space of all $\widetilde{K}$-finite vectors (resp.\ smooth vectors, analytic vectors).
Then $\hrep{H}_{\widetilde{K}}$, $\hrep{H}^\infty$ and $\hrep{H}^\omega$ are $\widetilde{\lie{g}}$-modules.
It is well-known that any $(\widetilde{\lie{g}}, \widetilde{K})$-module of finite length has such a realization by the Harish-Chandra subquotient theorem.

For a $(\widetilde{\lie{g}}, \widetilde{K})$-module $V$ of finite length, we set $V^\infty \coloneq \hrep{H}^\infty$
and $V^\omega \coloneq \hrep{H}^\omega$, where $\hrep{H}$ is a Hilbert globalization of $V$, i.e.\ $\hrep{H}_{\widetilde{K}}\simeq V$.
$V^\infty$ and $V^\omega$ are independent of the choice of $\hrep{H}$.

We shall consider the ranks of the restrictions of $\hrep{H}$, $\hrep{H}^\infty$, $\hrep{H}^\omega$ and $\hrep{H}_{\widetilde{K}}$ to $G_\RR$ and $\lie{g}$.
At first, we give a remark about Hilbert globalizations of a $(\widetilde{\lie{g}}, \widetilde{K})$-module.
By Proposition \ref{prop:RankEasy}, a Hilbert globalization with finer topology has smaller ranks.
In fact, the ranks depend strongly on the choice of a Hilbert globalization as follows.

\begin{proposition}\label{prop:RankDependGlobalization}
	Let $V$ be an irreducible $(\widetilde{\lie{g}}, \widetilde{K})$-module.
	There exists a Hilbert globalization $\hrep{H}$ of $V$ such that $\Rrankmin(\hrep{H}|_{G_\RR}) = 0$.
	If, moreover, $V$ is unitarizable, there exists a Hilbert globalization $\hrep{H}$ of $V$ such that $\Rrankmax(\hrep{H}|_{G_\RR}) = 0$.
\end{proposition}

\begin{proof}
	The first assertion follows from Proposition \ref{prop:ExistenceQuotient}
	and the second assertion from Proposition \ref{prop:ExistenceQuotientUnitary}, by taking the dual.
\end{proof}

We shall state easy consequences of the results in the previous section.

\begin{proposition}\label{prop:BranchingEasy}
	Let $\hrep{H}$ be an irreducible admissible Hilbert representation of $\widetilde{G}_\RR$.
	\begin{enumerate}
		\item $\hrep{H}_{\widetilde{K}}|_{\lie{g}}$ and $\hrep{H}^\omega|_{G_\RR}$ are equi-rank.
		\item $\Rrankmin(\hrep{H}|_{G_\RR}) \leq \Rrankmax(\hrep{H}|_{G_\RR}) \leq \Rrankmax(\hrep{H}^\infty|_{G_\RR}) \leq \Rrank(\hrep{H}^\omega|_{G_\RR}) \leq \Rrank(\hrep{H}_{\widetilde{K}}|_{\lie{g}})$.
		\item $\glrank(\hrep{H}|_{G_\RR}) = \glrank(\hrep{H}^\infty|_{G_\RR}) = \glrank(\hrep{H}^\omega|_{G_\RR}) = \glrank(\hrep{H}_{\widetilde{K}}|_{\lie{g}})$
	\end{enumerate}
\end{proposition}

\begin{remark}
	$\hrep{H}^\infty|_{G_\RR}$ may not coincide with $(\hrep{H}|_{G_\RR})^\infty$.
\end{remark}

\begin{proof}
	Since $\hrep{H}$ is irreducible, $\hrep{H}_{\widetilde{K}}$ is an irreducible $\widetilde{\lie{g}}$-module
	and $\hrep{H}^\omega$ is an irreducible representation of $\widetilde{G}_\RR$.
	Theorem \ref{thm:RankFiltration} shows that $\hrep{H}_{\widetilde{K}}|_{\lie{g}}$ is equi-rank.
	Set $k\coloneq \Rrankmin(\hrep{H}^\omega|_{G_\RR})$.
	By Theorem \ref{thm:RankFiltration}, $\rfilt{k}{\hrep{H}^\omega|_{G_\RR}}$
	is $\widetilde{\lie{g}}$-stable, and hence the closure $\overline{\rfilt{k}{\hrep{H}^\omega|_{G_\RR}}}$ in $\hrep{H}^\omega$ is $\widetilde{G}_\RR$-stable.
	Hence we have $k = \Rrankmax(\hrep{H}^\omega|_{G_\RR})$.

	The second and third assertions follow from Proposition \ref{prop:RankEasyConti} and the fact that $\hrep{H}_{\widetilde{K}}$ is dense in $\hrep{H}$, $\hrep{H}^\infty$ and $\hrep{H}^\omega$.
\end{proof}

The ranks are preserved by the translation functor.
More generally, we have the following results.

\begin{proposition}\label{prop:TranslationStable}
	Let $V$ be an irreducible $(\widetilde{\lie{g}}, \widetilde{K})$-module and let $F$ be a finite-dimensional $(\widetilde{\lie{g}}, \widetilde{K})$-module.
	\begin{enumerate}
		\item $\glrank(W|_{\lie{g}}) = \glrank(V|_{\lie{g}})$ for any non-zero submodule or quotient $W$ of $F\otimes V$.
		\item $W$ is equi-rank and $\Rrank(W|_{\lie{g}}) = \Rrank(V|_{\lie{g}})$ for any non-zero submodule $W$ of $F\otimes V$.
		\item $\Rrankmax(W|_{\lie{g}}) = \Rrank(V|_{\lie{g}})$ for any non-zero quotient $W$ of $F\otimes V$.
	\end{enumerate}
\end{proposition}

\begin{remark}
	In the third assertion, $W|_{\lie{g}}$ may not be equi-rank.
	In fact, if $V$ is the irreducible highest weight module of $\classicalg{sl}(2,\CC)$ with highest weight $-2$ and $F$ is the irreducible module with highest weight $2$, then a quotient of $V\otimes F$ has a finite-dimensional non-zero submodule.
	Then a split Cartan subgroup $G_\RR$ of $\classicalG{SL}(2,\RR)$ gives an example. 
\end{remark}

\begin{proof}
	We shall consider a non-zero submodule $W$ of $V$.
	The proof for quotients is similar, so we omit the details.

	By Corollary \ref{cor:TranslationSupportNonFinGen}, $\Rrank((V\otimes F)|_{\lie{g}}) = \Rrank(V|_{\lie{g}})$.
	Hence the second assertion follows from Proposition \ref{prop:RankEasy}.

	Let $\iota\colon W\rightarrow F\otimes V$ denote the inclusion.
	Using the canonical isomorphism $\Hom_{\widetilde{\lie{g}}, \widetilde{K}}(W, F\otimes V) \simeq \Hom_{\widetilde{\lie{g}}, \widetilde{K}}(F^*\otimes W, V)$,
	we obtain a non-zero homomorphism $\iota' \colon F^* \otimes W\rightarrow V$.
	Since $V$ is irreducible, $\iota'$ is surjective.
	By Proposition \ref{prop:RankEasy} and Corollary \ref{cor:TranslationSupportNonFinGen}, we obtain
	$\glrank(W|_{\lie{g}}) = \glrank((F\otimes V)|_{\lie{g}}) = \glrank(V|_{\lie{g}})$.
\end{proof}

We need to modify Proposition \ref{prop:TranslationStable} for continuous representations because of the difference between Propositions \ref{prop:RankEasy} and \ref{prop:RankEasyConti}.
We state the result only for direct summands.

\begin{proposition}
	Let $V$ be an irreducible $(\widetilde{\lie{g}}, \widetilde{K})$-module and let $F$ be a finite-dimensional $(\widetilde{\lie{g}}, \widetilde{K})$-module.
	Let $W$ be a direct summand of $F\otimes V$.
	\begin{enumerate}
		\item $W^\omega|_{G_\RR}$ is equi-rank and $\Rrank(W^\omega|_{G_\RR}) = \Rrank(V^\omega|_{G_\RR})$.
		\item We have $\Rrankmax(W^\infty|_{G_\RR}) = \Rrankmax(V^\infty|_{G_\RR})$ and $\Rrankmin(W^\infty|_{G_\RR}) = \Rrankmin(V^\infty|_{G_\RR})$.
	\end{enumerate}
\end{proposition}

\begin{proof}
	We shall show the second assertion.
	The first assertion is shown by the same way.

	By Corollary \ref{cor:TranslationSupportNonFinGen}, $\rfilt{k}{F\otimes V^\infty|_{\lie{g}}} = F\otimes \rfilt{k}{V^\infty|_{\lie{g}}}$ for any $k$.
	This implies
	\begin{align*}
		\Rrankmin(F\otimes V^\infty|_{G_\RR}) &= \Rrankmin(V^\infty|_{G_\RR}),\\
		\Rrankmax(F\otimes V^\infty|_{G_\RR}) &= \Rrankmax(V^\infty|_{G_\RR}).
	\end{align*}

	Recall that any homomorphism from a $(\widetilde{\lie{g}}, \widetilde{K})$-module $V_1$ to a $(\widetilde{\lie{g}}, \widetilde{K})$-module $V_2$ canonically extends to a continuous $\widetilde{G}_\RR$-linear map from $V_1^\infty$ to $V_2^\infty$ (see e.g.\ \cite[Theorem 11.6.7]{Wa92_real_reductive_II}).
	This implies that $W^\infty$ is a direct summand of $(F\otimes V)^\infty = F\otimes V^\infty$.
	Moreover, using the isomorphism $\Hom_{\lie{g}, K}(W, F\otimes V) \simeq \Hom_{\lie{g}, K}(F^*\otimes W, V)$,
	we obtain a surjective homomorphism $F^* \otimes W\rightarrow V$,
	and similarly an injective homomorphism $V\rightarrow F^* \otimes W$.
	Hence $V^\infty$ is a direct summand of $F^*\otimes W^\infty$.
	Therefore, Proposition \ref{prop:RankEasyConti} shows the desired equalities.
\end{proof}

\subsection{Rank on isotypic component}

In this subsection, we introduce a new rank, which comes from the $\univcent{g}$-action on each $K$-type.
The rank is more `stable' than the previous ones.
Retain the notation in the previous subsection.

For a reductive algebraic group $H$, let $\widehat{H}$ denote the set of all equivalence classes of irreducible $H$-module.
For $\tau \in \widehat{H}$ and an $H$-module $V$, we denote by $V(\tau)$ the isotypic component.

\begin{definition}
	Let $V$ be a $(\lie{g}, K)$-module.
	We define
	\begin{align*}
		\Krankmin(V) &\coloneq \min\set{\dim(\Variety(\Ann_{\univcent{g}}(V(\tau)))): \tau \in \widehat{K}}, \\
		\Krankmax(V) &\coloneq \max\set{\dim(\Variety(\Ann_{\univcent{g}}(V(\tau)))): \tau \in \widehat{K}}.
	\end{align*}
	If $\Krankmin(V) = \Krankmax(V)$, we write $\Krank(V) \coloneq \Krankmin(V)$.
\end{definition}

Obviously, we have
\begin{align*}
	\Rrankmin(V) \leq \Krankmin(V), \quad \Rrankmax(V) \leq \Krankmax(V).
\end{align*}
The equality $\Krankmin(V) = \Krankmax(V)$ often holds in the context of the branching problem.

\begin{proposition}\label{prop:EquiKrank}
	Let $V$ be an irreducible $(\lie{\widetilde{g}}, K)$-module.
	Then $\Krankmin(V|_{\lie{g}, K}) = \Krankmax(V|_{\lie{g}, K})$ holds.
\end{proposition}

\begin{proof}
	Let $\tau, \tau' \in \widehat{K}$ such that $V(\tau)$ and $V(\tau')$ are non-zero.
	By completely reducibility, $V(\tau)$ and $V(\tau')$ are irreducible $\univ{\widetilde{g}}^{K}\otimes \univ{k}$-modules.
	Then we can take a finite-dimensional $G$-submodule $F\subset \univ{g}$ such that
	$V(\tau') \subset \univ{\widetilde{g}}^{K}F\cdot V(\tau)$.
	Note that any finitely generated submodule of the $(\univ{\widetilde{g}}^{K}, \univ{\widetilde{g}}^{K})$-bimodule $\univ{\widetilde{g}}$
	is also finitely generated as a right $\univ{\widetilde{g}}^{K}$-module (see \cite[Lemma 7.39]{Ki23}).
	Hence there exists a finite-dimensional $G$-submodule $F'\subset \univ{\widetilde{g}}$ such that
	$V(\tau') \subset F'\univ{\widetilde{g}}^{K}\cdot V(\tau) = F'\cdot V(\tau)\subset F'\cdot \univ{g}V(\tau)$.
	Therefore, by Corollary \ref{cor:TranslationSupportNonFinGen}, we have
	\begin{align*}
		\dim(\Variety(\Ann_{\univcent{g}}(V(\tau')))) &\leq \glrank(\univ{g}V(\tau)) \\
		&=\dim(\Variety(\Ann_{\univcent{g}}(V(\tau)))).
	\end{align*}
	Swapping $\tau$ for $\tau'$, we obtain the converse,
	and hence we have shown the theorem.
\end{proof}

Since there exists a projection onto any $K_\RR$-isotypic component in each globalization, $\Krank$ does not depend on the choice of a globalization.
This is the reason why we consider the rank.

\begin{lemma}\label{lem:PairingKrank}
	Let $V$ and $W$ be $(\lie{g}, K)$-modules and $(\cdot, \cdot) \colon V\times W\rightarrow \CC$ a $(\lie{g}, K)$-invariant pairing.
	Suppose that the pairing is non-degenerate with respect to $W$, i.e.\ the induced map $W\rightarrow V^*$ is injective.
	Then we have $\Krankmax(W) \leq \Krankmax(V)$ and $\Krankmin(W) \leq \Krankmin(V)$.
\end{lemma}

\begin{proof}
	Let $\tau \in \widehat{K}$.
	By assumption, there is an injective $\univcent{g}$-homomorphism $W(\tau)\rightarrow V(\tau^*)^*$.
	Note that $\Ann_{\univcent{g}}(V(\tau^*)^*) = {}^t\Ann_{\univcent{g}}(V(\tau^*))$.
	This implies
	\begin{align*}
		\dim(\Variety(\Ann_{\univcent{g}}(W(\tau))))
		&\leq \dim(\Variety(\Ann_{\univcent{g}}(V(\tau^*)^*))) \\
		&= \dim(\Variety(\Ann_{\univcent{g}}(V(\tau^*))))
	\end{align*}
	and we have shown the lemma.
\end{proof}

\begin{theorem}
	Let $\hrep{H}$ be an irreducible admissible Hilbert representation of $\widetilde{G}_\RR$ and $W\subset (\hrep{H}^\infty)_{K}$ a non-zero $(\lie{\widetilde{g}}, K)$-module.
	Then we have $\Krankmax(W|_{\lie{g}, K}) = \Krankmax((\hrep{H}^*)_{\widetilde{K}}|_{\lie{g}, K})$.
	The same is true for $\Krankmin$.
	In particular, we have $\Krankmax(W|_{\lie{g}, K}) = \Krankmin(W|_{\lie{g}, K})$, and they do not depend on the choice of $W$.
\end{theorem}

\begin{proof}
	Set $U\coloneq (\hrep{H}^*)_{\widetilde{K}}$.
	Then $U$ is an irreducible $(\lie{\widetilde{g}}, \widetilde{K})$-module.
	Consider the restriction $(\cdot, \cdot)\colon W\times U\rightarrow \CC$ of the canonical pairing between $\hrep{H}$ and $\hrep{H}^*$.
	Then $(\cdot, \cdot)$ is $(\lie{\widetilde{g}}, K)$-invariant.

	The pairing induces two homomorphisms $\varphi\colon W\rightarrow U^*$ and $\varphi'\colon U\rightarrow W^*$.
	Since $U$ is dense in $\hrep{H}^*$, $\varphi$ is injective, and
	since $U$ is irreducible, $\varphi'$ is also injective.
	Hence $(\cdot, \cdot)$ is non-degenerate with respect to both factors.
	The assertion follows from Lemma \ref{lem:PairingKrank} and Proposition \ref{prop:EquiKrank}.
\end{proof}

As a consequence, we see that $\Krank$ is preserved under taking the dual.

\begin{corollary}\label{cor:DualKrank}
	Let $V$ be an irreducible $(\lie{\widetilde{g}}, \widetilde{K})$-module.
	Then $\Krank(V|_{\lie{g}, K}) = \Krank((V^*)_{\widetilde{K}}|_{\lie{g}, K})$ holds.
\end{corollary}

\subsection{Comparison of ranks}

In this section, we treat the problem whether the ranks are equal or not.
We have no complete answer to this problem, so we give partial results.
Retain the notation in the previous subsection.

Fix a $\theta$-stable Cartan subalgebra $\lie{h}_\RR$ of $\lie{g}_\RR$ and $W$ denotes the Weyl group of $\lie{g}$.
Set $\lie{h}_c\coloneq \lie{h}^{-\theta} \oplus \sqrt{-1}\lie{h}^{\theta}$.
Then any $\lie{h}$-weight in $\univ{\widetilde{g}}$ is real valued on $\lie{h}_c$.
For $\lambda \in \lie{h}$, we denote by $\Re(\lambda)$ and $\Im(\lambda)$ the real part and the imaginary part of $\lambda$, respectively.
In other words, $\lambda = \Re(\lambda) + \sqrt{-1}\Im(\lambda)$, and $\Re(\lambda)$ and $\Im(\lambda)$ are real valued on $\lie{h}_c$.
Since $\lie{h}_c$ is $W$-stable, the decomposition $\lambda = \Re(\lambda) + \sqrt{-1}\Im(\lambda)$ induces that of $[\lambda]$.
We set $\Re([\lambda]) \coloneq [\Re(\lambda)]$ and $\Im([\lambda]) \coloneq [\Im(\lambda)]$.

Write $q\colon \lie{h}^*\rightarrow \lie{h}^*/W$ for the quotient map.
For a $\lie{g}$-module or a Hilbert representation $V$ of $G_\RR$ with an infinitesimal character,
we denote by $\Lambda(V)$ the infinitesimal character.

We have proved the following lemma in \cite[Proposition 14]{Ki24}.
The lemma is proved by using a realization of a unitary representation by unitary parabolic induction.

\begin{lemma}\label{lem:BoundUnitaryRep}
	Let $\hrep{H}$ be an irreducible unitary representation of $G_\RR$ with infinitesimal character and $K$-type $\tau \in \widehat{K}$.
	Then there exists a constant $C$ depending only on $\tau$ such that $\|\Re(\Lambda(\hrep{H}))\| \leq C$.
\end{lemma}

The following result is useful to check $\Rrank(V) = \Krank(V)$.
Write $R$ for the set of all $\lie{h}$-weights in $\univ{\widetilde{g}}$.
Note that $R$ is the lattice in $\lie{h}_c^*$ generated by all $\lie{h}$-weights in $\lie{\widetilde{g}}$.

\begin{lemma}\label{lem:AffineToEquiRank}
	Let $\hrep{H}$ be an irreducible unitary representation of $\widetilde{G}_\RR$
	and $V\subset (\hrep{H}^\infty)_{K}$ a finitely generated $(\lie{\widetilde{g}}, K)$-submodule.
	Assume that
	\begin{enumerate}
		\item $V|_{\lie{g}, K}$ has a small Cartan subalgebra $\lie{a} = \lie{h}/\lie{m}$ and
		\item $\lie{a}^*\subset \lie{h}^*$ is spanned by some elements in $R$.
	\end{enumerate}
	Then $\Rrank(V|_{\lie{g}, K}) = \Krank(V|_{\lie{g}, K})$ holds.
\end{lemma}

\begin{proof}
	First we shall show the assertion when $V = \univ{\widetilde{g}}V^P$ for some $P \in \Ass_{\univcent{g}}(V)$.
	Take $\lambda_0 \in \lie{h}^*$ such that $\Variety(P) = q(\lie{a}^* + \lambda_0)$.
	Since $V$ is equi-rank and $V= \univ{\widetilde{g}}V^P$, there exists $\mu \in R$
	such that $\Variety(Q) = q(\lie{a}^* + \lambda_0 + \mu)$ for any $Q \in \Ass_{\univcent{g}}(V)$ by Lemma \ref{lem:AnnTranslation}.

	Let $\tau \in \widehat{K}$ and show that $\Ass_{\univcent{g}}(V(\tau))$ is a finite set.
	Fix $C > 0$ in Lemma \ref{lem:BoundUnitaryRep} for $\tau$.
	Let $Q \in \Ass_{\univcent{g}}(V(\tau))$,
	and take $\mu \in R$ such that $\Variety(Q) = q(\lie{a}^* + \lambda_0 + \mu)$.
	By Proposition \ref{prop:ExistenceQuotientUnitary}, $V^Q$ has a unitarizable irreducible quotient $W$.
	By Lemma \ref{lem:BoundUnitaryRep}, we have $\|\Re(\Lambda(W))\| \leq C$.
	This implies
	\begin{align*}
		\Re(\lie{a}^* + \lambda_0 + \mu) \cap \set{\lambda \in \lie{h}^*_c : \|\lambda\| \leq C} \neq \emptyset.
	\end{align*}

	Set $\lie{a}^*_c\coloneq \lie{a}^* \cap \lie{h}^*_c$.
	Then $\lie{a}^* = \lie{a}^*_c \oplus \sqrt{-1} \lie{a}^*_c$, and we have
	\begin{align}
		(\lie{a}^*_c + \Re(\lambda_0) + \mu) \cap \set{\lambda \in \lie{h}^*_c : \|\lambda\| \leq C} \neq \emptyset. \label{eqn:Bound}
	\end{align}
	By assumption, the sublattice $\lie{a}^*_c\cap R$ of $R$ spans $\lie{a}^*_c$.
	Hence $R/(\lie{a}^*_c \cap R)$ is discrete in $\lie{h}^*_c / \lie{a}^*_c$.
	This shows that \eqref{eqn:Bound} holds for only finitely many $Q \in \Ass_{\univcent{g}}(V(\tau))$.
	Hence $\Ass_{\univcent{g}}(V(\tau))$ is a finite set.

	Next we shall consider the general case.
	Since $V$ is noetherian, there are finitely many ideals $Q_1, \ldots, Q_r \in \Ass_{\univcent{g}}(V(\tau))$ such that
	\begin{align*}
		\Ass_{\univcent{g}}\left(\left(\sum_i \univ{\widetilde{g}}V^{Q_i}\right)(\tau)\right) = \Ass_{\univcent{g}}(V(\tau)).
	\end{align*}
	By using the first case, $\Ass_{\univcent{g}}(V(\tau))$ is a finite set.
	By Proposition \ref{prop:UnitaryRadicalIdeal}, $\Ann_{\univcent{g}}(V(\tau)) = \bigcap_{Q\in \Ass_{\univcent{g}}(V(\tau))} Q$ holds.
	Therefore we obtain $\Rrank(V) = \Krank(V)$.
\end{proof}

As we have shown in Theorem \ref{thm:Affinity}, the assumption in Lemma \ref{lem:AffineToEquiRank} holds
if $V$ is irreducible.
In particular, we obtain the following theorem.

\begin{theorem}\label{thm:KrankRrankGK}
	Let $V$ be a unitarizable irreducible $(\lie{\widetilde{g}}, \widetilde{K})$-module.
	Then we have $\Krank(V|_{\lie{g}, K}) = \Rrank(V|_{\lie{g}, K})$.
\end{theorem}

\begin{corollary}\label{cor:KrankRrankGK}
	Let $\hrep{H}$ be an irreducible unitary representation of $\widetilde{G}_\RR$.
	Assume that there exists $P \in \Ass_{\univcent{g}}(\hrep{H}^\infty)$ such that $\dim(\Variety(P)) = \Rrankmin(\hrep{H}^\infty|_{G_\RR})$ and $\Variety(P)$ is the image by $q$ of a translation of a subspace spanned by some elements in $R$.
	Then $\hrep{H}^\infty|_{G_\RR}$ is equi-rank and $\Rrank(\hrep{H}^\infty|_{G_\RR}) = \Rrank(\hrep{H}_K|_{\lie{g}, K})$ holds.
\end{corollary}

\begin{proof}
	Since $(\hrep{H}^\infty)^P$ is a closed in $\hrep{H}^\infty$ and $K_\RR$-stable,
	we can take $0 \neq v \in ((\hrep{H}^\infty)^P)_K$.
	Applying Lemma \ref{lem:AffineToEquiRank} to $\univ{\widetilde{g}}v$, we obtain the corollary.
\end{proof}

The assumption of affinity is fulfilled if $\Rrankmin(\hrep{H}^\infty|_{G'_\RR}) = 0$.
In this case, we have proved the equi-rank property for non-unitary representations in \cite[Theorem 9]{Ki24}.

\begin{theorem}\label{thm:Rank0}
	Let $V$ be an irreducible $(\lie{\widetilde{g}}, \widetilde{K})$-module.
	Then $\Rrankmin(V^\infty|_{G_\RR}) = 0$ if and only if $\Rrank(V|_{\lie{g}, K}) = 0$.
\end{theorem}

Recall $\Rrankmin(V^\infty|_{G_\RR}) \leq \Rrank(V|_{\lie{g}, K})$.
If $\Rrank(V|_{\lie{g}, K}) \leq 1$, the rank $\Rrankmin(V^\infty|_{G_\RR})$ is $0$ or $1$.
Hence, as a consequence of Theorem \ref{thm:Rank0}, we obtain the following result.

\begin{proposition}\label{prop:Rank1}
	Let $V$ be an irreducible $(\lie{\widetilde{g}}, \widetilde{K})$-module.
	If $\Rrank(V|_{\lie{g}, K}) \leq 1$, then $V^\infty|_{G_\RR}$ is equi-rank and $\Rrank(V^\infty|_{G_\RR}) = \Rrank(V|_{\lie{g}, K})$.
\end{proposition}

In Corollary \ref{cor:BoundByRank}, we will show $\Rrank(V|_{\lie{g}, K}) \leq \rank_\RR(G_\RR)$.
Hence the assumption in Proposition \ref{prop:Rank1} always holds if $\rank_\RR(G_\RR) = 1$.

\subsection{Associated variety}\label{subsection:AV}

In this subsection, we deal with a relation between the rank and the associated variety of a $(\lie{g}, K)$-module.
We assume that all negative parts of a filtration of a vector space (e.g.\ $\CC$-algebra or its module) are zero.
We return to the setting of generalized pairs.
Let $(\alg{A}, G)$ be a generalized pair with connected reductive algebraic group $G$.

We shall recall the definition of associated varieties of $\lie{g}$-modules.
We refer the reader to \cite{BoBr82}, \cite{Vo89} and \cite{Ko98_discrete_decomposable_3}.
Let $V$ be a finitely generated module of a filtered $\CC$-algebra $\alg{R}$ and $F$ an exhaustive filtration of $V$ satisfying that $\alg{R}_i F_j(V)\subset F_{i+j}(V)$ for any $i, j\in \NN$.
Then $\gr_F V$ is a $\gr \alg{R}$-module.
The filtration $F$ is said to be a good filtration if $\gr_F V$ is finitely generated.
It is well-known that any finitely generated module has a good filtration.

Suppose that $F$ is a good filtration and $\gr \alg{R}$ is a finitely generated commutative $\CC$-algebra
and an integral domain.
The variety $\AV(V)\coloneq \Variety(\Ann_{\gr \alg{R}}(\gr_F V))$ is called the associated variety of $V$.
Note that $\AV(V)$ does not depend on the choice of a good filtration.
Since $\gr_F V$ is a graded module, $\AV(V)$ is a $\CC^\times$-variety.

The associated variety is preserved by tensoring with finite-dimensional modules.
See \cite[Lemma 4.1]{BoBr82}.

\begin{fact}\label{fact:TranslationAV}
	Let $V$ be a finitely generated $\lie{g}$-module and $F$ a finite-dimensional $\lie{g}$-module.
	Then $\AV(V\otimes F) = \AV(V)$ holds.
\end{fact}

The following proposition follows from Fact \ref{fact:TranslationAV}.
The proposition is a prototype of Theorems \ref{thm:RankFiltration} and \ref{thm:Affinity}.
The proof is the same as \cite[Theorem 3.7]{Ko98_discrete_decomposable_3}.

\begin{proposition}\label{prop:AmoduleAV}
	Let $V$ be an irreducible $\alg{A}$-module.
	Then $\AV(\univ{g}v)$ is independent of $0\neq v \in V$.
\end{proposition}

\begin{definition}
	Let $\alg{R}$ be a filtered $\CC$-algebra whose associated graded algebra is an integral domain and a finitely generated $\CC$-algebra.
	For a non-zero $\alg{R}$-module $V$, if $\AV(\alg{R}v)$ does not depend on the choice of $0\neq v \in V$, we set $\AV(V)\coloneq \AV(\alg{R}v)$ taking $0\neq v \in V$.
\end{definition}

Suppose that $\alg{A}$ has at most countable dimension.
Let $\theta$ be an involution of $G$ and $K$ a (connected and finite) covering of $(G^\theta)_0$.
Then $(\alg{A}, K)$ forms a generalized pair.
Fix a Cartan subalgebra $\lie{h}$ of $\lie{g}$.
Let $W$ denote the Weyl group of $\lie{g}$.
Write $q\colon \lie{h}^* \rightarrow \lie{h}^*/W$ for the quotient map.

For an affine $L$-variety $X$ of a connected reductive algebraic group $L$,
we denote by $X\GIT L$ the GIT quotient of $X$, i.e.\ $X\GIT L$ is the affine variety with the coordinate ring $\rring{X}^L$.
The isomorphism $S(\lie{g})^G\simeq S(\lie{h})^W$ induces an isomorphism $\lie{g}^*\GIT G \xrightarrow{\simeq} \lie{h}^* \GIT  W$.
Note that there exists a natural bijection between $\lie{h}^*\GIT W$ and $\lie{h}^*/W$ as sets.
Write $\widetilde{q}\colon \lie{g}^*\rightarrow \lie{g}^*\GIT G\simeq \lie{h}^*\GIT W$ for the quotient morphism induced from the inclusion $S(\lie{g})^G\rightarrow S(\lie{g})$.

For an $(\alg{A}, K)$-module $V$, since we can take a $K$-stable good filtration $F$ of $V$ (i.e.\ each $F_i(V)$ are $K$-stable), the associated variety $\AV(V)$ is $K$-stable and contained in $(\lie{g}/\lie{k})^*$.

\begin{lemma}\label{lem:RestrictionGoodFiltration}
	Let $V$ be a finitely generated $(\lie{g}, K)$-module with a $K$-stable good filtration $F$ and a $K$-type $\tau \in \widehat{K}$.
	Then $V(\tau)$ is a finitely generated $\univcent{g}$-module, and the filtration $F'$ of $V(\tau)$ defined by $F'_i(V(\tau))\coloneq (F_i V)(\tau)$
	is a good filtration as a $\univcent{g}$-module.
\end{lemma}

\begin{proof}
	For the first assertion, see \cite[Theorem 3.4.1]{Wa88_real_reductive_I}.
	It suffices to show that $(\gr_F V)(\tau)$ is finitely generated as an $S(\lie{g})^G$-module.
	It is well-known that $(\gr_F V)(\tau)$ is finitely generated as an $S(\lie{g}/\lie{k})^K$-module
	(\cite[Lemma 3.4.3]{Wa88_real_reductive_I}).
	Note that $\lie{k}\subset S(\lie{g})$ acts on $(\gr_F V)(\tau)$ trivially.
	Moreover, $S(\lie{g}/\lie{k})^K$ is a finitely generated $S(\lie{g})^G$-module via the natural homomorphism $S(\lie{g})^G\rightarrow S(\lie{g}/\lie{k})^K$ (\cite[Lemma 3.4.4]{Wa88_real_reductive_I}).
	Therefore $(\gr_F V)(\tau)$ is finitely generated as a $\univcent{g}$-module.
\end{proof}

For an ideal $I$ of a filtered algebra $\alg{R}$, we provide $I$ with the filtration induced from that of $\alg{R}$.
Then $\gr I$ is an ideal of $\gr \alg{R}$, and $\gr(\alg{R})/\gr(I)$ is isomorphic to $\gr(\alg{R}/I)$.

\begin{lemma}\label{lem:GrAffine}
	Let $P$ be a prime ideal of $S(\lie{h})$ such that $\Variety(P)$ is a translation of a subspace $\lie{a}^*$ of $\lie{h}^*$.
	Then $\Variety(\gr (P\cap S(\lie{h})^W)) = q(\lie{a}^*)$ holds, where $\gr$ is taken for the filtration defined by degree.
\end{lemma}

\begin{proof}
	It is easy to see $\Variety(\gr P) = \lie{a}^*$.
	Set $Q\coloneq \bigcap_{w \in W} w P$.
	We shall show $\Variety(\gr Q) = W\cdot \lie{a}^*$.
	Note that taking $\gr$ preserves the inclusion relation.
	$\gr Q\subset \bigcap_{w \in W} w\cdot \gr P$ is clear.
	By $Q\supset \prod_{w \in W} wP$, we have
	\begin{align*}
		\gr Q \supset \gr\left(\prod_{w \in W} wP\right) \supset \prod_{w\in W} w\cdot \gr(P).
	\end{align*}
	Therefore we obtain $\Variety(\gr Q) = W\cdot \lie{a}^*$.

	Since the $W$-action on $S(\lie{h})$ is completely reducible, we have
	\begin{align*}
		\gr(P \cap S(\lie{h})^W) = \gr(Q \cap S(\lie{h})^W) = \gr(Q^W)
		= \gr(Q)^W = \gr(Q) \cap S(\lie{h})^W.
	\end{align*}
	This shows the desired equality $\Variety(\gr(P \cap S(\lie{h})^W)) = q(\Variety(\gr(Q))) = q(\lie{a}^*)$.
\end{proof}

\begin{theorem}\label{thm:RankAV}
	Let $V$ be an irreducible $(\alg{A}, K)$-module.
	Take a small Cartan subalgebra $\lie{a}=\lie{h}/\lie{m}$ for $V|_{\lie{g}}$.
	Then $\overline{\widetilde{q}(\AV(V|_{\lie{g}}))} = q(\lie{a}^*)$ holds.
	Moreover, we have $\dim(\AV(V|_{\lie{g}})\GIT K) = \Rrank(V|_{\lie{g}})$.
\end{theorem}

\begin{proof}
	Fix a finitely generated $(\lie{g}, K)$-submodule $U$ of $V$ with a $K$-stable good filtration.
	By Theorem \ref{thm:Affinity}, for any $P \in \Ass_{\univcent{g}}(U)$, there exists $\lambda \in \lie{h}^*$ such that $\Variety(P) = q(\lie{a}^* + \lambda)$.
	Since $\gr U$ is finitely generated, there exist $\tau_1, \tau_2, \ldots, \tau_r \in \widehat{K}$
	such that $\gr U$ is generated by $\bigoplus_{i}(\gr U)(\tau_i) = \bigoplus_{i}\gr(U(\tau_i))$.
	By Lemma \ref{lem:RestrictionGoodFiltration}, we have
	\begin{align*}
		\widetilde{q}(\AV(U)) &= \Variety(\Ann_{S(\lie{g})^G}(\gr U)) \\
		&= \Variety\left(\bigcap_{i} \Ann_{S(\lie{g})^G}(\gr(U(\tau_i)))\right) = \bigcup_i \AV(U(\tau_i)|_{\univcent{g}}).
	\end{align*}
	Lemma \ref{lem:GrAffine} shows $\AV(U(\tau_i)|_{\univcent{g}}) = q(\lie{a}^*)$ for any $i$.
	Therefore we obtain the desired equality $\overline{\widetilde{q}(\AV(U))} = q(\lie{a}^*)$.

	Set $I\coloneq \sqrt{\Ann_{S(\lie{g})}(\gr U)}$.
	As in Lemma \ref{lem:RestrictionGoodFiltration}, $S(\lie{g}/\lie{k})^K$ is finitely generated as an $S(\lie{g})^G$-module.
	Hence $(S(\lie{g})/I)^K$ is finitely generated over the subalgebra $S(\lie{g})^G/(S(\lie{g})^G \cap I)$.
	This implies $\dim(\AV(V|_{\lie{g}})\GIT K) = \dim(\overline{\widetilde{q}(\AV(V|_{\lie{g}}))}) = \Rrank(V|_{\lie{g}})$
\end{proof}

As a consequence of Theorem \ref{thm:RankAV}, we obtain a upper bound of the rank.
Fix a connected reductive Lie group $G_\RR$ corresponding to the pair $(\lie{g}, K)$.
Recall that $\dim(S(\lie{g}/\lie{k})^K) = \rank_{\RR}(G_\RR)$, where $\rank_{\RR}(\cdot)$ means the real rank.

\begin{corollary}\label{cor:BoundByRank}
	Let $V$ be an irreducible $(\alg{A}, K)$-module.
	Then $\Rrank(V|_{\lie{g}}) \leq \rank_{\RR}(G_\RR)$ holds.
\end{corollary}

It is hard to compute $\AV(V|_{\lie{g}})$ directly.
We give a lower bound of $\AV(V|_{\lie{g}})$ using the associated variety for $\alg{A}$.
Suppose that $\alg{A}$ is a filtered $\CC$-algebra such that the homomorphism $\univ{g}\rightarrow \alg{A}$ is a morphism of filtered algebras, and $\gr \alg{A}$ is an integral domain and a finitely generated $\CC$-algebra.
Universal enveloping algebras and algebras of twisted differential operators are typical examples.
Let $X$ denote the affine variety with coordinate ring $\gr\alg{A}$.
Then the homomorphism $S(\lie{g})\simeq \gr\univ{g} \rightarrow \gr \alg{A}$ induces a morphism $\res_{\lie{g}}\colon X\rightarrow \lie{g}^*$.

The following proposition is a generalization of \cite[Theorem 3.1]{Ko98_discrete_decomposable_3},
which is stated for the discretely decomposable case.
The proof is essentially the same.

\begin{proposition}\label{prop:LowerBoundAV}
	Let $V$ be an irreducible $\alg{A}$-module.
	Then $\overline{\res_{\lie{g}}(\AV(V))} \subset \AV(V|_{\lie{g}})$ holds.
	If, moreover, $V$ is an $(\alg{A}, K)$-module and the filtration of $\alg{A}$ is $K$-stable, then $\Rrank(V|_{\lie{g}})\geq \dim(\overline{\res_{\lie{g}}(\AV(V))}\GIT K)$.
\end{proposition}

\begin{proof}
	Take a non-zero vector $v \in V$ and define a good filtration $F$ of $V$ via $F_i(V)\coloneq \alg{A}_i v$ ($i \in \NN$).
	Similarly, define a good filtration $F'$ of $V'\coloneq \univ{g}v$ via $F'_i(V)\coloneq \univfilt{i}{g} v$ ($i \in \NN$).
	Write $\overline{v}$ for the image of $v$ via the inclusion $F'_0(V)\hookrightarrow \gr_{F'} {V'}$.
	The inclusion $V'\hookrightarrow V$ induces a homomorphism $\varphi\colon \gr_{F'} V' \rightarrow \gr_F V$.
	Then the $\gr \alg{A}$-module $\gr_F V$ is generated by $\varphi(\overline{v})$.
	Therefore we have
	\begin{align*}
		\Ann_{S(\lie{g})}(\gr_F V) = \Ann_{S(\lie{g})}(\varphi(\overline{v})) \supset \Ann_{S(\lie{g})}(\overline{v}) = \Ann_{S(\lie{g})}(\gr_{F'} V').
	\end{align*}
	This shows the desired inclusion $\overline{\res_{\lie{g}}(\AV(V))} \subset \AV(V|_{\lie{g}})$.
	The second assertion follows from the first one and Theorem \ref{thm:RankAV}.
\end{proof}

Let $(\widetilde{\lie{g}}, \widetilde{K})$ be a pair constructed from a connected reductive Lie group $\widetilde{G}_\RR$.
Suppose that $(\lie{g}, K)$ is a subpair of $(\widetilde{\lie{g}}, \widetilde{K})$,
that is, $\lie{g}\subset \widetilde{\lie{g}}$ and $K\subset \widetilde{K}$.
In the case, it is conjectured in \cite[Conjecture 5.11]{Ko11} that $\res_{\lie{g}}(\AV(V)) = \AV(V|_{\lie{g}})$ holds
when $V$ is an irreducible $(\widetilde{\lie{g}}, \widetilde{K})$-module and $V|_{\lie{g}, K}$ is discretely decomposable.
If the conjecture holds without discrete decomposability, the rank is computed from the associated variety $\AV(V)$.
We shall state the conjecture explicitly.

\begin{conjecture}\label{conj:EnhancedKobayashi}
	Let $V$ be an irreducible $(\widetilde{\lie{g}}, \widetilde{K})$-module.
	Then $\overline{\res_{\lie{g}}(\AV(V))} = \AV(V|_{\lie{g}})$ and $\Rrank(V|_{\lie{g}}) = \dim(\overline{\res_{\lie{g}}(\AV(V))} \GIT K)$ hold.
\end{conjecture}

The following result for the real rank one case supports the conjecture.
Similarly to Proposition \ref{prop:Rank1}, this is proved by using the result for the discrete decomposability.

\begin{proposition}\label{prop:Rank1ToConj}
	Let $V$ be an irreducible $(\widetilde{\lie{g}}, \widetilde{K})$-module.
	Suppose $\rank_{\RR}(G_\RR) = 1$ and $(\widetilde{G}_\RR, G_\RR)$ is a symmetric pair.
	Take a small Cartan subalgebra $\lie{a} = \lie{h}/\lie{m}$ for $V|_{\lie{g}}$.
	Then $\overline{\widetilde{q}(\res_{\lie{g}}(\AV(V)))} = q(\lie{a}^*)$ and $\dim(\overline{\res_{\lie{g}}(\AV(V))}\GIT K) = \Rrank(V|_{\lie{g}})$ hold.
\end{proposition}

\begin{proof}
	As in Theorem \ref{thm:RankAV}, the second equality follows from the first one.
	Recall that we have seen $\widetilde{q}(\res_{\lie{g}}(\AV(V)))\subset q(\lie{a}^*) \leq \rank_{\RR}(G_\RR) = 1$ in Corollary \ref{cor:BoundByRank} and Proposition \ref{prop:LowerBoundAV}.
	Note that $q(\lie{a}^*)$ is irreducible.
	Hence it is enough to show the assertion when $\dim(\overline{\widetilde{q}(\res_{\lie{g}}(\AV(V))}) = 0$.

	Since $\AV(V)$ is defined by a graded ideal, $\widetilde{q}(\res_{\lie{g}}(\AV(V))$ is $\CC^\times$-stable.
	From this and $\dim(\widetilde{q}(\res_{\lie{g}}(\AV(V))) = 0$, we have $\widetilde{q}(\res_{\lie{g}}(\AV(V)) = \set{0}$.
	This implies that $\res_{\lie{g}}(\AV(V)$ is contained in the nilpotent cone in $(\lie{g}/\lie{k})^*$.
	This shows that $V|_{\lie{g}}$ is discretely decomposable and hence $\Rrank(V|_{\lie{g}}) = 0$ as we have proved in \cite[Theorem 21]{Ki24}.
\end{proof}

\begin{remark}
	The assumption that $(\widetilde{G}_\RR, G_\RR)$ is a symmetric pair is used to apply \cite[Theorem 21]{Ki24}.
	Proposition \ref{prop:Rank1ToConj} holds under the assumption in \cite[Theorem 21]{Ki24} (and $\rank_{\RR}(G_\RR) = 1$).
	We omit the details.
\end{remark}

\subsection{Cartan subalgebra}

We shall check the condition of Theorem \ref{thm:AffinityAnn} under some assumptions.
Let $\widetilde{G}_\RR$, $\widetilde{K}_\RR$, $\widetilde{K}$, $G_\RR$, $K_\RR$ and $K$ be as in Subsection \ref{subsection:GKBasic}.

If $(\widetilde{G}_\RR, G_\RR)$ is a symmetric pair, then the condition of Theorem \ref{thm:AffinityAnn} is fulfilled for any irreducible $\lie{\widetilde{g}}$-module $V$.
More generally, we have the following criterion.

\begin{theorem}
	Suppose that $\lie{g}$ is spherical in $\lie{\widetilde{g}}$, that is, there exists a Borel subalgebra $\lie{\widetilde{b}}$ of $\lie{\widetilde{g}}$ such that $\lie{\widetilde{b}} + \lie{g} = \lie{\widetilde{g}}$.
	Then $V|_{\lie{g}}$ has a Cartan subalgebra in the sense of Definition \ref{def:Cartan} for any irreducible $\lie{\widetilde{g}}$-module $V$.
\end{theorem}

\begin{proof}
	Since $V$ is irreducible, $\Ann_{\univ{\widetilde{g}}}(V)$ is a primitive ideal.
	By Duflo's theorem (\cite{Du77_primitive_ideal}), there exists an irreducible highest weight module $V'$ of $\widetilde{\lie{g}}$ with respect to $\lie{\widetilde{b}}$ such that $\Ann_{\univ{\widetilde{g}}}(V') = \Ann_{\univ{\widetilde{g}}}(V)$.
	By $\lie{\widetilde{b}} + \lie{g} = \lie{\widetilde{g}}$, $V'|_{\lie{g}}$ is finitely generated.
	In fact, $V'|_{\lie{g}}$ is generated by a highest weight vector.
	Therefore $V$ satisfies the assumption of Theorem \ref{thm:AffinityAnn} and we have shown the assertion.
\end{proof}

For irreducible $(\widetilde{\lie{g}}, \widetilde{K})$-modules, the assumption can be weakened.
To do so, we recall the results by H.\ Yamashita \cite[Theorems I and IV]{Ya94}.
The following criterion is useful to find a $\widetilde{\lie{g}}$-module whose restriction to $\lie{g}$ is finitely generated.

\begin{fact}\label{fact:Yamashita1}
	Let $V$ be a finitely generated $\widetilde{\lie{g}}$-module.
	If $\AV(V)\cap \lie{g}^\perp = \set{0}$, then $V|_{\lie{g}}$ is finitely generated.
\end{fact}

\begin{fact}\label{fact:Yamashita2}
	Suppose that $\lie{g}_\RR$ is real spherical in $\lie{\widetilde{g}}_\RR$, i.e.\ there exists a minimal parabolic subalgebra $\lie{\widetilde{p}}_\RR$ of $\lie{\widetilde{g}}_\RR$ such that $\lie{\widetilde{p}}_\RR + \lie{g}_\RR = \lie{\widetilde{g}}_\RR$.
	Let $V$ be an irreducible $(\widetilde{\lie{g}}, \widetilde{K})$-module.
	Then $\AV(V)\cap \Ad^*(g)\lie{g}^\perp = \set{0}$ for some $g \in \widetilde{G}_\RR$.
\end{fact}

Combining these facts and Theorem \ref{thm:AffinityAnn}, we obtain the following criterion.

\begin{theorem}\label{thm:AffinityRealSpherical}
	Suppose that $\lie{g}_\RR$ is real spherical in $\lie{\widetilde{g}}_\RR$.
	Then $V|_{\lie{g}, K}$ has a Cartan subalgebra in the sense of Definition \ref{def:Cartan} for any irreducible $(\lie{\widetilde{g}}, \widetilde{K})$-module $V$.
\end{theorem}

\begin{proof}
	Take $g \in \widetilde{G}_\RR$ such that $\AV(V)\cap \Ad^*(g)\lie{g}^\perp = \set{0}$ as in Fact \ref{fact:Yamashita2}.
	Define an irreducible $\lie{\widetilde{g}}$-module $(\pi, V')$ by $V' = V$ as a vector space and $\pi(X)v = \Ad(g)(X)v$ ($v \in V', X \in \lie{\widetilde{g}}$).
	Then it is easy to see that $\Ann_{\univ{\widetilde{g}}}(V') = \Ann_{\univ{\widetilde{g}}}(V)$ and $\AV(V') = \Ad^*(g)^{-1}\AV(V)$.
	Hence we have $\AV(V')\cap \lie{g}^\perp = \set{0}$.
	This shows that $V'|_{\lie{g}}$ is finitely generated by Fact \ref{fact:Yamashita1}.
	Applying Theorem \ref{thm:AffinityAnn} to $V$ and $V'$, we have proved the theorem.
\end{proof}

We shall state similar results to Theorem \ref{thm:RankAV} and Proposition \ref{prop:LowerBoundAV} for $\Variety(\gr (\Ann_{\univcent{g}}(V)))$.
They follow from $\gr(\Ann_{\univ{g}}(V))^G = \gr(\Ann_{\univ{g}}(V)^G)$ easily, so we omit the proof.
See also the proofs of Theorem \ref{thm:RankAV} and Proposition \ref{prop:LowerBoundAV}.

\begin{proposition}\label{prop:AssociatedVarietyAnn}
	Let $V$ be an irreducible $\lie{\widetilde{g}}$-module and set $I\coloneq \Ann_{\univ{\widetilde{g}}}(V)$.
	Suppose that $V|_{\lie{g}}$ has a Cartan subalgebra $\lie{c}$ in the sense of Definition \ref{def:Cartan}.
	\begin{enumerate}
		\item $\overline{\widetilde{q}(\Variety(\gr(I\cap \univ{g})))} = \Variety(\gr(I\cap \univcent{g})) = q(\lie{c}^*)$.
		\item $\dim(\overline{\res_{\lie{g}}(\Variety(\gr I))} \GIT G) \leq \glrank(V)$.
	\end{enumerate}
	Here $\widetilde{q}\colon \lie{g}^* \rightarrow \lie{g}^* \GIT G$, $q\colon \lie{h}^* \rightarrow \lie{h}^*\GIT W$, $\res_{\lie{g}}\colon \lie{\widetilde{g}}^*\rightarrow \lie{g}^*$ are maps defined in the previous subsection.
\end{proposition}

It is well-known that $\Variety(\gr I)$ is the closure of one nilpotent coadjoint orbit in $\widetilde{\lie{g}}^*$.
See \cite[Theorem 9.3]{Ja04} and references therein.
Moreover, $\Variety(\gr I)$ is a Poisson $G$-variety and $\res_{\lie{g}}$ is the moment map on the variety.
In the case, $\overline{\res_{\lie{g}}(\Variety(\gr I))}$ and related varieties are studied,
and the Cartan subalgebra and the Weyl group for such a Poisson $G$-variety are constructed in \cite{Lo09}.
If the conjectural equality $\overline{\res_{\lie{g}}(\Variety(\gr I))} = \Variety(\gr(I\cap \univ{g}))$
holds (Conjecture \ref{conj:EnhancedKobayashi}), the Cartan subalgebra $\lie{c}$ for $V$ is constructed from the Poisson $G$-variety $\Variety(\gr I)$.
Similarly, we expect that the small Cartan subalgebra $\lie{a}$ for $V$ is constructed from a Lagrangian subvariety of the nilpotent coadjoint orbit.

\subsection{Abelian case}

Let us consider the abelian subgroup case.
We have treated the abelian case in Subsection \ref{subsection:AbelianGeneral}.
We shall state a special result in the branching problem.

Let $G_\RR$ be a connected semisimple Lie group with finite center
and $\theta$ a Cartan involution of $G_\RR$.
Set $K_\RR\coloneq G_\RR^\theta$ and let $K$ denote the complexification of $K_\RR$.
Let $\lie{a}_\RR$ be an maximal abelian subspace in $\lie{g}_\RR^{-\theta}$,
and write $A_\RR$ for the analytic subgroup of $G_\RR$ corresponding to $\lie{a}_\RR$.
Then the algebraic group $A\coloneq \overline{\Ad(A_\RR)}$ in the Zariski topology of $\Aut(\lie{g})$ is a complex torus.
Note that $(\univ{g}, A)$ is a generalized pair.

For a $\lie{g}$-module $V$, we set
\begin{align*}
	\lie{g}[V]\coloneq \set{X \in \lie{g}: \CC X \text{ acts locally finitely on }V}.
\end{align*}
Then it is known that $\lie{g}[V]$ is a subalgebra of $\lie{g}$.
By definition, the $\lie{g}[V]$-action on $V$ is locally finite.
See e.g.\ \cite[Theorem 1.16]{Zu12} and references therein.
If $V$ is a $(\lie{g}, K)$-module, $\lie{g}[V]$ contains $\lie{k}$ by definition.

\begin{theorem}\label{thm:AbelianGK}
	Let $V$ be an irreducible $(\lie{g}, K)$-module.
	Suppose that no simple factor of $\lie{g}$ acts on $V$ locally finitely.
	Then $\Ass_{\univ{a}}(V) = \set{(0)}$ holds.
	In other words, the $\univ{a}$-module $V|_{\lie{a}}$ is torsion-free.
\end{theorem}

\begin{proof}
	Assume that the $\univ{a}$-module $V|_{\lie{a}}$ is not torsion-free.
	By Corollary \ref{cor:AbelianGen} for $(\univ{g}, A)$, there exists a non-zero subalgebra $\lie{m}$ of $\lie{a}$ such that the $\lie{m}$-action on $V$ is locally finite.
	Hence $\lie{g}[V]$ contains $\lie{k}$ and $\lie{m}$.
	Since $\lie{m}$ is non-zero and contained in $\lie{g}^{-\theta}$,
	$\lie{g}[V]$ contains a simple factor of $\lie{g}$.
	This contradicts the assumption of the theorem.
	We have shown the theorem.
\end{proof}

Unitarity is essential in Theorem \ref{thm:AbelinUnitary}.
To show Theorem \ref{thm:AbelianGK}, we use the property that $\lie{a}_\RR$ is split abelian instead of the unitarity.
In the $(\lie{g}, K)$-module case, we do not need the algebraicity in Theorem \ref{thm:AbelinUnitary}.
In fact, the restriction of torsion-free module to a subalgebra is also torsion-free.

\begin{corollary}\label{cor:AbelianGK}
	Let $V$ be an irreducible $(\lie{g}, K)$-module.
	Suppose that no simple factor of $\lie{g}$ acts on $V$ locally finitely.
	For any abelian subspace $\lie{a}'_\RR$ in $\lie{g}_\RR^{-\theta}$, the $\univ{a'}$-module $V|_{\lie{a'}}$ is torsion-free.
\end{corollary}

%% file: ref.bbl
\def\cprime{$'$} \def\cprime{$'$}